\documentclass[10pt]{amsart}

\usepackage{amsmath}
\usepackage{amssymb}
\usepackage{bm}
\usepackage{graphicx}
\usepackage{color}
\definecolor{wine}{rgb}{0.5,0,0.15} %
\usepackage{hyperref}
\hypersetup{colorlinks=true, linkcolor=blue, citecolor=magenta, urlcolor=wine}
\usepackage{url}
\usepackage{algpseudocode}
\usepackage{fancyhdr}
\usepackage{mathtools}
\usepackage{tikz-cd}
\usepackage{xy}
\input xy
\xyoption{all}
\usepackage{stmaryrd}
\usepackage{calrsfs}

\newtheorem*{maintheorem*}{Main Theorem}
\newtheorem{theorem}{Theorem}[section]
\newtheorem{prop}[theorem]{Proposition}

\newtheorem{lem}[theorem]{Lemma}
\newtheorem{cor}[theorem]{Corollary}
\theoremstyle{definition}
\newtheorem{defn}[theorem]{Definition}
\newtheorem{rem}[theorem]{Remark}
\newtheorem{example}[theorem]{Example}

\numberwithin{equation}{section}

\newcommand{\cc}{\mathbb{C}}
\newcommand{\ff}{\mathbb{F}}
\newcommand{\nn}{\mathbb{N}}
\newcommand{\pp}{\mathbb{P}}

\newcommand{\qq}{\mathbb{Q}}
\newcommand{\rr}{\mathbb{R}}
\newcommand{\zz}{\mathbb{Z}}

\providecommand\ldb{\llbracket}
\providecommand\rdb{\rrbracket}

\newcommand{\gp}{\mathcal{G}}

\newcommand{\supp}{\textsf{supp}}

\newcommand{\uu}{\mathcal{U}}



\begin{document}
	
\mbox{}
\title{An unrestricted notion of the finite factorization property}

\author{Jonathan Du}
\address{PRIMES-USA\\MIT\\Cambridge, MA 02139}
\email{jonathan.cx.du@gmail.com}

\author{Felix Gotti}
\address{Department of Mathematics\\MIT\\Cambridge, MA 02139}
\email{fgotti@mit.edu}

\date{\today}

\begin{abstract}
    A nonzero element of an integral domain (or an element of a commutative cancellative monoid) is called atomic if it is a unit or can be written as a finite product of irreducible elements (also called atoms). In this paper, we introduce and investigate an \emph{unrestricted} version of the finite factorization property, extending the work on unrestricted UFDs carried out by Coykendall and Zafrullah in 2004. An integral domain is said to have the unrestricted finite factorization (U-FF) property if every atomic element has only finitely many factorizations, or equivalently, if its atomic submonoid has the finite factorization (FF) property. We position the U-FF property within the hierarchy of classical finiteness conditions, showing that every IDF domain is a U-FFD but not conversely. Then we analyze the behavior of the U-FF property under standard constructions, including localization, polynomial extensions, and the $D+M$ construction: in each case, we find sufficient conditions under which the U-FF property is inherited by the construction. We also prove that, as it is the case with the IDF property, the U-FF property does not ascend to polynomial extensions. Then we extend the known characterization of FFDs as atomic IDF domains by arguing that the FF property and the U-FF property are equivalent over the class of nearly atomic domains (i.e., integral domains having a nonzero principal ideal whose nonzero elements are atomic). These results demonstrate that the U-FF property behaves in many ways analogously to the IDF property. We conclude taking a look at the class of all unrestricted UFDs.
\end{abstract}

\maketitle

\bigskip
\section{Introduction}
\label{sec:intro}

The study of factorizations lies at the heart of modern commutative algebra and number theory. Understanding how elements in a ring or monoid decompose into irreducibles and how such decompositions may fail to be unique has shaped the algebraic landscape from the nineteenth century to the present. The origin of this line of inquiry can be traced to the failure of unique factorization in cyclotomic rings of integers, which led Kummer~\cite{eK1847} to introduce his \emph{ideal numbers} and Dedekind~\cite{rD1871} to formalize the theory of ideals. Their work not only addressed Lam\'e's erroneous proof of Fermat's Last Theorem~\cite{gL1847}, but also inaugurated the arithmetic of algebraic number fields. The resulting notion of the ideal class group provided the first systematic measure of the deviation from unique factorization, a concept that remains fundamental in algebraic number theory and arithmetic geometry.
\smallskip

A century later, in 1960, Carlitz~\cite{lC60} deepened this connection by proving that the ring of integers of a number field is half-factorial if and only if its class group has size at most~$2$ (being a UFD if and only if its class group has size $1$). Half-factoriality has been systematically studied since the eighties by Chapman, Coykendall, et al. (see, for instance, \cite{CC00}). Carlitz's result established that half-factoriality for rings of integers is completely governed by the class number. Thus, Dedekind domains with cyclic class groups of order~$2$ represent the simplest nontrivial deviation from having unique factorization, although any two factorizations of the same nonzero nonunit have the same length (i.e., number of irreducible factors). Carlitz's result marked a turning point between classical ideal theory and modern factorization theory, introducing a quantitative viewpoint that would later evolve into the systematic study of length sets and other arithmetical invariants.
\smallskip

The significance of factorization theory extends far beyond its historical origins in the analysis of Lam\'e's flawed proof of Fermat's Last Theorem and Kummer's introduction of ideal numbers. Today, factorization theory serves as a unifying framework connecting diverse areas of algebra, number theory, geometry, and combinatorics. The following overview summarizes some of its principal applications and contemporary directions.

\medskip
\subsection{Algebraic Number Theory and Arithmetic Geometry}

The earliest and most profound applications of factorization theory lie in algebraic number theory. Kummer's ideal numbers and Dedekind's subsequent formalization of ideals provided the conceptual foundation for the study of \emph{ideal class groups}, whose structure measures the failure of unique factorization in rings of algebraic integers. Class groups play a central role in understanding the arithmetic of algebraic number fields, influencing the solvability of Diophantine equations, the structure of algebraic curves, and the arithmetic of algebraic varieties (see \cite{AT67}). The modern theory of global fields, divisor class groups, and Picard groups in algebraic geometry continues to depend fundamentally on these ideas (see~\cite{sL94}).

\medskip
\subsection{Commutative Algebra and Module Theory}

Within commutative algebra, the relevance of factorization-theoretic finiteness properties is that they provide arithmetical surrogates for Noetherian finiteness and allow one to extend many structural and arithmetic results beyond the Noetherian or UFD settings, to vast classes of non-Noetherian rings and monoids where unique factorization fails but divisibility remains sufficiently well-behaved. The following are among the most classical factorization-theoretic finiteness properties studied in the literature:
\begin{itemize}
    \item the atomic property, which was introduced and studied by Cohn~\cite{pC68} in 1968 -- every nonzero nonunit factors into finitely many irreducibles;
    \smallskip
    
    \item the IDF property, which was introduced and studied by Grams and Warner~\cite{GW75} in 1975 -- every nonzero element has finitely many irreducible divisors up to associates;
    \smallskip
    
    \item the BF property, which was introduced and studied by Anderson, Anderson, and Zafrullah~\cite{AAZ90} in 1990 -- every nonzero element has a nonempty finite length set; and
    \smallskip
    
    \item the FF property, which was also introduced and studied by Anderson, Anderson, and Zafrullah~\cite{AAZ90} in 1990 -- every nonzero element has a nonempty finite set of factorizations, up to order and associates.
\end{itemize}
Such finiteness conditions often act as workable substitutes for Noetherianity, enabling a systematic extension of classical finiteness results, including the FF property (and the BF property) to broad classes of non-Noetherian rings (see, for instance, \cite{AG22}). On the other hand, the property of $v$-Noetherianity provides quantitative measures of how much arithmetic structure persists in non-Noetherian settings, which have proved instrumental in analyzing divisor-closed submonoids, $v$-ideals, and integral closures (see \cite{GH06a} and \cite[Chapters 2--4]{GH06}).

\medskip
\subsection{Additive Combinatorics and Zero-Sum Theory}

A particularly fruitful interaction occurs between factorization theory and additive combinatorics. Through the work of Geroldinger, Halter-Koch, and Schmid it was recognized that sets of lengths of factorizations in Krull monoids correspond to zero-sum sequences over the subset of the class group consisting of classes containing prime divisors (see, for instance, \cite{GS16,fHK92a}). This correspondence allows one to translate algebraic questions about non-unique factorizations into combinatorial problems on sumsets and sequences. Classical invariants such as the Davenport constant control the maximal degree of non-uniqueness of factorizations (see~\cite{aG09}). Consequently, techniques from additive number theory, including the Cauchy-Davenport~\cite{alC1813} and Olson theorems~\cite{jO69}, have become central tools in modern factorization research.

\medskip
\subsection{Algebraic Geometry and Invariant Theory}

In algebraic geometry, factorization properties describe the arithmetic and geometric behavior of coordinate rings and local rings. They play a crucial role in determining normality, factoriality, and divisor class groups of varieties (see \cite[Chapter~4]{CLS11}). In invariant theory, establishing whether an invariant ring under a group action is factorial (or a UFD) is essential to describing quotient varieties and their coordinate structures (see \cite{MFK94}). In this sense, factorization properties provide an algebraic language for geometric phenomena.

\medskip
\subsection{Computational and Algorithmic Applications}

From a computational perspective, factorization theory underlies key algorithms in computer algebra and cryptography. Algorithms for factoring integers and polynomials are fundamental to computational number theory and to public-key crypto-systems such as RSA (see \cite[Chapters 1--3]{hC93}). Understanding patterns of non-unique factorizations can inform the design of secure arithmetic systems that resist decomposition-based attacks (see \cite{CP05}). Moreover, in symbolic computation, determining whether a given integral domain or commutative monoid satisfies atomicity, factoriality, or the FF property has algorithmic relevance in computer algebra systems such as \textsc{Magma} or \textsc{Singular} (see \cite{HL16}), and it serves as a structural finiteness/termination criterion for factorization procedures (see \cite{BHL17}), which underlies practical algorithms implemented in computer algebra systems such as \textsc{Singular} (see the computational framework in Greuel-Pfister).
\smallskip

In summary, factorization theory provides a cohesive framework that connects number theory, algebra, geometry, and combinatorics. From Kummer's ideal numbers to contemporary studies on Puiseux monoids and unrestricted finite factorization, the theory continues to illuminate how algebraic structures decompose, how finiteness conditions govern their arithmetic, and how these decompositions influence problems ranging from Diophantine equations to computational cryptography.

\medskip
\subsection{The Bounded and Finite Factorization Properties}

Let $M$ be a cancellative commutative monoid. We say that an element $a \in M$ is atomic if it is a unit or can be written as a finite product of irreducibles (also called atoms), and we let $M_A$ denote the submonoid of~$M$ consisting of all atomic elements. In their seminal paper~\cite{AAZ90}, Anderson, Anderson, and Zafrullah first investigated integral domains with the BF/FF property, and systematically studied them in connection to atomicity, ACCP, Krullness, and Noetherianity. Their inclusion diagram (Figure~\ref{fig:AAZ diagram}) remains a guiding reference for modern factorization theory.
\begin{center}
    \begin{figure}[ht]
        \begin{tikzcd} %
            \textbf{ UF }    \arrow[r, Rightarrow] \arrow[red, r, Leftarrow, "/"{anchor=center,sloped}, shift left=1.7ex] 
            \arrow[d, Rightarrow, shift right=1ex] \arrow[red, d, Leftarrow, "/"{anchor=center,sloped}, shift left=1ex]
            & \textbf{ HF }    \arrow[d, Rightarrow, shift right=1ex] \arrow[red, d, Leftarrow, "/"{anchor=center,sloped}, shift left=1ex] \\
            \textbf{ FF }   \arrow[r, Rightarrow, shift right=1ex] \arrow[red, r, Leftarrow, "/"{anchor=center,sloped}, shift left=1ex]
            & \textbf{ BF }   \arrow[r, Rightarrow] \arrow[red, r, Leftarrow, "/"{anchor=center,sloped}, shift left=1.7ex] 
            & \textbf{ ACCP } \arrow[r, Rightarrow] \arrow[red, r, Leftarrow, "/"{anchor=center,sloped}, shift left=1.7ex] 
            & \textbf{ AT }
        \end{tikzcd}
        \caption{The implications in the diagram show the inclusions among subclasses of atomic monoids (AT stands for the class of atomic monoids). The (red) marked arrows emphasize that none of the shown implications are reversible.}
        \label{fig:AAZ diagram}
    \end{figure}
\end{center}

\medskip
\subsection{The Unrestricted Finite Factorization Property}

In 2004, Coykendall and Zafrullah~\cite{CZ04} introduced the notion of an unrestricted unique factorization domain (U--UFD), which is an integral domain where every atomic element, rather than every nonzero element, has a unique factorization. This insight shifted attention from the entire integral domain to its atomic submonoid and inspired the broader idea that any classical factorization property can be extended \emph{unrestrictedly} by requiring it to hold only on the multiplicative submonoid consisting of all atomic elements. Formally, for any atomic property~$\mathcal{P}$, a monoid~$M$ is said to have property~$\mathcal{P}$ \emph{unrestrictedly} if its atomic submonoid~$M_A$ has property~$\mathcal{P}$. Motivated by this perspective, we introduce and study the unrestricted finite factorization property, which is the central property in the scope of this paper.

\begin{defn}
    A commutative monoid (or integral domain) is said to have the \emph{U-FF} property if every atomic element has finitely many factorizations.
\end{defn}

Equivalently, a monoid has the U-FF property if its atomic submonoid is an FFM. This definition provides a natural generalization of both the FF and the IDF properties. Indeed, every monoid having the FF property (and hence every atomic IDF monoid) also has the U-FF property (although the converse does not hold). The U-FF property fits naturally within the web of classical finiteness conditions. It weakens the FF property and the IDF property, since every monoid having the IDF property will necessarily have the U-FF property (we will present this as Proposition~\ref{prop:IDF are U-FF}). Our first goal is therefore to locate the U-FF property precisely among other finiteness conditions such as the BF, IDF, and MCD--finite properties.

\medskip
\subsection{Outline of the Paper}

Section~\ref{sec:background} collects notation and background on commutative monoids, factorization invariants, and the IDF and MCD--finite properties, allowing the paper to be self-contained. 
\smallskip

Section~\ref{sec:properties weaker than the FF} offers a preliminary study of the U-FF property, analyzing its relationship to other generalizations of the FF property, and provides several explicit examples and counterexamples illustrating the independence of the BF, IDF, and MCD--finite properties. 
\smallskip

Section~\ref{sec:localization} is devoted to the study of the behavior of the U-FF property under localization. The first investigation of factorization properties under localization was carried out by Anderson, Anderson, and Zafrullah~\cite{AAZ92}. They showed that the BF and FF properties ascend from $R$ to $R_S$ if the extension $R \subseteq R_S$ is inert (e.g., localizations at multiplicative sets generated by primes), while this is not the case for the IDF property; instead, the IDF property ascends provided that $S$ is a splitting multiplicative set (i.e., every nonzero element of $R$ has the form $as$, where $s\in S$ and $aR$ is $S$-saturated). Here we prove that, similarly to the IDF property, the U-FF property also does not necessarily ascend when localizing at a multiplicative set generated by primes, but it does ascend under splitting multiplicative sets and descend under splitting multiplicative sets generated by primes.
\smallskip

Section~\ref{sec:D+M} investigates the behavior of the U-FF property under the $D+M$ construction, introduced and studied by Gilmer~\cite[Appendix II]{rG68} for valuation domains. Let $T$ be an integral domain, and let $K$ and $M$ be a subfield of $T$ and a nonzero maximal ideal of $T$, respectively, such that $T=K+M$. For a subring $D$ of $K$, set $R=D+M$. We establish precise conditions under which the U-FF property transfers from $T=K+M$ to subrings of $T$ of the form $R=k+M$, where $k$ is a subfield of $K$.
\smallskip

Section~\ref{sec:ascent to polynomial rings} examines the ascent of the U-FF property to polynomial extensions. First, we prove that the U-FF property ascends to polynomial extensions over the class of MCD-finite domains. %
Then we argue that the MCD-finite condition is not superfluous by constructing an explicit example of an integral domain having the U-FF property whose polynomial extension fails to have the U-FF property, thereby showing that the property does not ascend in general to polynomial extensions.
\smallskip

A classical result proved by Anderson, Anderson, and Zafrullah in~\cite[Theorem~5.1]{AAZ90} states that the IDF property and the FF property are equivalent on the class of atomic domains. In this direction, the recent paper by the authors and Hong shows that the class of nearly atomic domains is strictly larger than the class of atomic domains, and then they prove that the IDF property and the FF property are still equivalent on this larger class. Section~\ref{sec:a characterization of FFDs via near atomicity}, which is a brief section, further extends the equivalence by showing that the U-FF property and the FF property are equivalent on the class of nearly atomic domains.
\smallskip

This study situates the unrestricted finite factorization property within the landscape of classical factorization theory, showing that it behaves robustly under many standard constructions (such as $D+M$ pullbacks) yet fails to ascend to polynomial rings, much like the classical IDF property. The examples presented here illustrate the subtle boundary between atomic and unrestricted arithmetic in integral domains and contribute to the ongoing effort to extend the scope of factorization theory beyond the atomic setting.

\bigskip
\section{Background}
\label{sec:background}

In this section, we fix the notation, terminology, and preliminary results required to make the paper self-contained.

\medskip
\subsection{General Notation}

As usual, $\zz$, $\qq$, $\rr$, and $\cc$ denote the sets of integers, rational numbers, real numbers, and complex numbers, respectively. We let $\nn$ and $\nn_0$ denote the sets of positive and nonnegative integers, respectively, and write $\pp$ for the set of prime numbers. For $p \in \pp$ and $n \in \nn$, we let $\ff_{p^n}$ denote the finite field of cardinality $p^n$. If $m,n \in \zz$ with $m \le n$, then we write the discrete interval from $m$ to $n$ as follows:
\[ 
    \ldb m,n \rdb := \{j \in \zz : m \le j \le n\}.
\]
For a subset $S \subseteq \rr$ and $r \in \rr$, we write
\[
    S_{\ge r} := \{s \in S : s \ge r\} \quad \text{and} \quad S_{> r} := \{s \in S : s > r\}.
\]

\medskip
\subsection{Commutative Monoids}

A pair $(S,\cdot)$, where $S$ is a set and `$\cdot$' is a binary operation on $S$, is called a \emph{semigroup} if `$\cdot$' is associative, in which case we say that $S$ is a semigroup under~$\cdot$. A semigroup $(S,\cdot)$ is \emph{commutative} if $s \cdot t = t \cdot s$ for all $s,t \in S$. Let $(S, \cdot)$ be a commutative semigroup. An element $r \in S$ is called \emph{cancellative} if, for all $s,t \in S$, the equality $r \cdot s = r \cdot t$ implies $s=t$. If every element of $S$ is cancellative, then $S$ itself is said to be \emph{cancellative}. As every monoid considered in this paper is both commutative and cancellative, we adopt the following convention.

\begin{defn}
    Throughout this paper, the single term \emph{monoid}\footnote{In the standard literature, a monoid is a semigroup with an identity element.} refers to a commutative and cancellative semigroup $(S,\cdot)$ with an element $1 \in S$, which we call the \emph{identity}, such that $1 \cdot s = s$ for all $s \in S$.
\end{defn}

For the rest of this section, let $(M, \cdot)$ be a monoid and let $1$ be \emph{the} identity element of $M$ (clearly, $M$ has exactly one identity element). As it is customary, we will write~$M$ to avoid the more cumbersome notation $(M,\cdot)$. A nonempty subset $N \subseteq M$ is a \emph{submonoid} if it is closed under multiplication and contains the identity element of~$M$. Given a subset $S \subseteq M$, the intersection of all the submonoids of $M$ containing~$S$ is denoted by $\langle S \rangle$ and called the \emph{submonoid generated} by~$S$. It is convenient to consider two abelian groups associated to~$M$. The abelian group consisting of all units (i.e., invertible elements) of~$M$ is denoted by $\uu(M)$ and is called the \emph{group of units} of $M$. We say that $M$ is \emph{reduced} if $\uu(M)$ is the trivial group. The quotient $M/\uu(M)$ is a reduced monoid, usually referred to as the \emph{reduced monoid} of~$M$. 
\smallskip

The abelian group consisting of all formal quotients of elements of~$M$ is denoted by~$\gp(M)$ and is referred to as the \emph{Grothendieck group} of $M$. Because $M$ is cancellative, it embeds into $\gp(M)$ in a natural way: indeed, $\gp(M)$ is the smallest abelian group (up to isomorphism) containing an isomorphic copy of~$M$. We say that $M$ is \emph{torsion-free} if $\gp(M)$ is a torsion-free abelian group. The \emph{rank} of~$M$ is the rank of the abelian group $\gp(M)$, viewed as a $\zz$-module. The monoid $M$ is \emph{linearly orderable} if there exists a total order relation $\preceq$ on $M$ such that if $a \prec b$ then $ac \prec bc$ for all $a,b,c \in M$. It follows by a result of Levi~\cite{fL13} that an abelian group is linearly orderable if and only if it is torsion-free, and from this one can deduce that~$M$ is a linearly orderable monoid if and only if it is a torsion-free monoid.

\medskip
\subsection{Divisibility and Atomicity}

For $b,c \in M$, we say that $c$ \emph{divides} $b$ if there exists $d \in M$ such that $b = cd$; we then write $c \mid_M b$. If both relations $b \mid_M c$ and $c \mid_M b$ hold, then $b$ and $c$ are said to be \emph{associates}, in which case we write $b \sim c$. This relation is an equivalence relation on~$M$. A submonoid~$N$ of~$M$ is said to be \emph{divisor-closed} if $a \in N$ and $b \mid_M a$ together imply $b \in N$. An element $d \in M$ is a \emph{common divisor} of a nonempty subset $S \subseteq M$ if $d \mid_M s$ for all $s \in S$. A common divisor $d$ of $S$ in~$M$ is called a \emph{maximal common divisor} (MCD) of~$S$ if the only common divisors of
\[
    S/d := \Big\{ \frac{s}d : s \in S \Big\}
\]
are the units of~$M$. Following Eftekhari and Khorsandi~\cite{EK18}, we say that the monoid~$M$ is \emph{MCD-finite} if every nonempty finite subset of $M$ has only finitely many MCDs up to associates.
\smallskip

An element $a \in M \setminus \uu(M)$ is an \emph{atom} or an \emph{irreducible element} if $a = uv$ with $u,v \in M$ implies $u \in \uu(M)$ or $v \in \uu(M)$. The set of atoms of $M$ is denoted by $\mathcal{A}(M)$. An element $b \in M$ is \emph{atomic} if it is a unit or can be written as a finite product of atoms (allowing repetitions). The set of all atomic elements of~$M$ is a submonoid, which we referred to as the \emph{atomic submonoid} of~$M$ and denoted by $A(M)$ (also written $M_A$). Following Cohn~\cite{pC68}, we say that $M$ is \emph{atomic} if $A(M) = M$. Following Boynton and Coykendall~\cite{BC15}, we say that $M$ is \emph{almost atomic} if for each $b \in M$ there exists an atomic element $a \in M$ such that $ab$ is also an atomic element of~$M$. Following Lockard~\cite{nL19}, we say that the monoid~$M$ is \emph{nearly atomic} if there exists $c \in M$ such that every element in the principal ideal $cM$ is atomic in $M$. An integral domain is \emph{almost atomic} (resp., \emph{nearly atomic}) provided that its multiplicative monoid is almost atomic (resp., nearly atomic). 
\smallskip

We say that $M$ is an \emph{irreducible-divisor-finite} (IDF) monoid if every element of~$M$ is divisible by only finitely many atoms up to associates, in which case we also say that $M$ has the IDF property. The question of whether the IDF property ascends to polynomial extensions was first posed by Anderson, Anderson, and Zafrullah~\cite{AAZ90} back in~1990 and negatively answered by Malcolmson and Okoh~\cite{MO09} two decades later. Positive results were obtained for various classes: the ascent of the IDF property holds for GCD domains~\cite{MO09}, for pre-Schreier domains~\cite{mZ17}, for MCD-finite domains~\cite{EK18}.

\medskip
\subsection{Factorizations}

Let $\mathsf{Z}(M)$ denote the free commutative monoid on the set of atoms of the reduced monoid $M/\uu(M)$, and let $\pi : \mathsf{Z}(M) \to M/\uu(M)$ be the unique monoid homomorphism fixing each element of $\mathcal{A}(M/\uu(M))$. For each $b \in M$, set $\mathsf{Z}(b) := \pi^{-1}(b\uu(M))$. The elements of $\mathsf{Z}(b)$ are called the \emph{factorizations} of~$b$ in~$M$. Clearly, $b$ has a factorization if and only if it is atomic. The monoid $M$ is a \emph{unique factorization monoid} (UFM) if $|\mathsf{Z}(b)| = 1$ for all $b \in M$. Following Anderson, Anderson, and Zafrullah~\cite{AAZ90}, we say that $M$ is a \emph{finite factorization monoid} (FFM) if every element of~$M$ has a nonempty finite set of factorizations; that is, $1 \le |\mathsf{Z}(b)| < \infty$ for all $b \in M$. The following result is well known, and we will generalize it in Section~\ref{sec:a characterization of FFDs via near atomicity}.

\begin{theorem} \cite[Theorem~2]{fHK92}
    A monoid is an FFM if and only if it is an atomic IDF monoid.
\end{theorem}

Therefore every FFM is an IDF monoid. Furthermore, if $M$ is an FFM and $S$ is a nonempty finite subset of~$M$ then after fixing $s \in S$, each MCD of~$S$ divides~$s$. Hence $S$ has only finitely many MCDs up to associates~\cite{fHK92}. Thus, every FFM is an MCD-finite monoid. 

\begin{rem}\label{rem:FF implies MCD-finite}
    Every FFM is an MCD-finite monoid.
\end{rem}

If $z = a_1 \cdots a_\ell \in \mathsf{Z}(M)$ with $a_1,\dots,a_\ell \in \mathcal{A}(M/\uu(M))$, then $|z| := \ell$ is the \emph{length} of the factorization~$z$. For each $b \in M$, the set
\[
	\mathsf{L}(b) := \{ |z| : z \in \mathsf{Z}(b) \}
\]
is called the \emph{length set} of~$b$. The monoid $M$ is a \emph{bounded factorization monoid} (BFM) if each element $b \in M$ has a nonempty finite length set, that is, $1 \le |\mathsf{L}(b)| < \infty$ for all $b \in M$. As an immediate consequence of the definition, we obtain the following.

\begin{rem}
    Every FFM is a BFM.
\end{rem}

The FF property was introduced in~\cite{AAZ90} along with the BF property, and both properties were considered in the setting of integral domains in the subsequent papers~\cite{AAZ92,AE99,AA10} by the same authors and Abidine. Both properties were first considered in the more general setting of cancellative commutative monoids by Halter-Koch~\cite{fHK92}. Anderson and Mullins dedicated a paper to the FF property~\cite{AM96}. For a recent survey of the bounded and finite factorization properties, see~\cite{AG22}.

\medskip
\subsection{Integral Domains}

Let $R$ be an integral domain. The multiplicative subset $R \setminus \{0\}$ is a monoid, which is denoted by $R^*$ and called the \emph{multiplicative monoid} of~$R$. As usual, $R^\times$ denotes the group of units of~$R$. An integral domain $R$ is called \emph{atomic} (an \emph{IDF domain}, an \emph{MCD-finite domain}, a \emph{BFD}, an \emph{FFD}, a \emph{UFD}) if its multiplicative monoid $R^*$ has the corresponding property. When convenient, we say that an atomic domain (resp., an IDF domain, an MCD-finite domain, a BFD, an FFD, a UFD) has the \emph{atomic} (resp., \emph{IDF}, \emph{MCD-finite}, \emph{BF}, \emph{FF}, or \emph{UF}) \emph{property}.

\smallskip
\subsubsection*{Localization, Domain Extensions, and Ideals}

Let $S$ be a multiplicative set of $R$, that is, a submonoid of $R^*$. We call $S$ \emph{saturated} if it is divisor-closed as a submonoid of $R^*$. Observe that every saturated multiplicative set is an intermediate monoid of the monoid extension $R^\times \subseteq R^*$. In the study of factorization properties under localization, the extension $R \mapsto R_S$ is central and, although most factorization properties do not transfer between $R$ and $R_S$, some of them will when the multiplicative set $S$ satisfies certain convenient properties, as shown in~\cite{AAZ92}.
\smallskip

An extension of integral domains $A \subseteq B$ is called \emph{inert} if for all $b_1, b_2 \in B$ such that $b_1 b_2 \in A\setminus\{0\}$ there exists $u \in B^\times$ such that $ub_1, u^{-1}b_2 \in A$. If $A \subseteq B$ is an inert extension, then $\mathcal{A}(A) \subseteq \mathcal{A}(B) \cup B^\times$ \cite[Lemma 1.1]{AAZ92}. If $S$ is a multiplicative set of $R$ generated by primes, then the extension $R \subseteq R_S$ is inert. A multiplicative set $S$ of $R$ is called a \emph{splitting multiplicative set} provided that for each $r \in R$ there exists $a \in R$ such that $r \in aS$ and $aR \cap sR = asR$ for all $s \in S$. If $S$ is a splitting multiplicative set, then the extension $R \subseteq R_S$ is inert.
\smallskip

We let $\mathcal{I}(R)$ denote the commutative semigroup consisting of all the ideals of~$R$ under standard multiplication of ideals. If $f \colon A \to B$ is a ring homomorphism, then we can define $(\cdot)^e \colon \mathcal{I}(A) \to \mathcal{I}(B)$ and $(\cdot)^c \colon \mathcal{I}(B) \to \mathcal{I}(A)$ as $I^e := f(I)B$ and $J^c := f^{-1}(J)$, respectively. The maps $(\cdot)^e$ and $(\cdot)^c$ are called the ideal extension and the ideal contraction associated to~$f$, respectively. In the special case when $f$ is the localization homomorphism $R \to R_S$ of $R$ at a multiplicative subset $S$, it immediately follows that for each ideal $I$ of $R$ the inclusion $I \subseteq (I^e)^c$ holds: $(I^e)^c$ is called the $S$-\emph{saturation} of $I$ and $I$ is called $S$-\emph{saturated} if $(I^e)^c = I$.

\smallskip
\subsubsection*{Monoid Algebras and Polynomial Rings} 

Let $M$ be a monoid, and let $x$ be an indeterminate. The \emph{monoid algebra} (or \emph{monoid domain}) of~$M$ over~$R$ consists of all the polynomial expressions with coefficients in $R$ and exponents in $M$ under polynomial-like addition and multiplication:
\[
    R[x;M] := \left\{ \sum_{i=1}^n c_i x^{q_i}  : (c_i, q_i) \in R \times M \text{ for every } i \in \ldb 1,n \rdb \right\}.
\]
The monoid algebra $R[x;M]$ is a commutative ring with identity. When there is no risk of confusion, we simply write $R[M]$ instead of $R[x;M]$. For every polynomial expression
\begin{equation} \label{eq:polynomial expression}
    f :=  \sum_{i=1}^n c_i x^{q_i} \in R[M],
\end{equation}
after combining equal exponents and omitting terms with zero coefficient, the set $\operatorname{supp} f := \{q_1,\dots,q_n\}$ is called the \emph{support} of~$f$. Now assume that $M$ is torsion-free. Since $R$ is an integral domain and $M$ is a cancellative torsion-free monoid, the monoid algebra $R[M]$ is also an integral domain, and its group of units is
\[
    R[M]^\times = \{\, d x^u : (d,u) \in R^\times \times \uu(M) \,\}.
\]
Moreover, the torsion-free and cancellative hypotheses ensure that there exists a total order~$\preceq$ on~$M$ such that $(M,\preceq)$ is a linearly ordered monoid, meaning that $c \prec d$ implies $b+c \prec b+d$ for all $b,c,d \in M$. Under this order, the representation of $f$ in \eqref{eq:polynomial expression} is unique with $q_n \succ \dots \succ q_1$, allowing us to define the \emph{degree} and \emph{order} of a nonzero element $f$ by
\[
    \deg f := q_n \quad \text{and} \quad \operatorname{ord} f := q_1.
\]

\smallskip
For monoid algebras, we mainly follow the terminology in Gilmer's book~\cite{rG84}, which the reader can consult for background information as it contains most of the relevant results in (commutative) monoid algebras established until 1984.
\smallskip

We conclude this section recalling some terminology about polynomials. Let $f = \sum_{i=0}^d c_i x^i \in R[x]$ be a polynomial of degree~$d$. The leading coefficient of $f$ is $c_d$, and we denote it here by $\text{lc}(f)$. The \emph{content} of $f$, denoted by $\mathfrak{c}(f)$, is the greatest common divisor (up to associates) of the coefficients $c_0, c_1, \dots, c_d$ in $R$, which means that
\[
    \mathfrak{c}(f) \sim \gcd(c_0, c_1, \dots, c_d),
\]
whenever such a greatest common divisor exists. Equivalently, $\mathfrak{c}(f)$ is any nonzero element $r \in R$ (unique up to associates) such that $r \mid_R c_i$ for every $i \in \ldb 0,d \rdb$ and such that every common divisor of $c_0,\dots,c_d$ divides $r$ in~$R$. A polynomial $f \in R[x]$ is called \emph{primitive} if $\mathfrak{c}(f) \in R^\times$, that is, if the coefficients of $f$ have no nonunit common divisor in~$R$.
\smallskip

\bigskip
\section{Generalizations of the Finite Factorization Property}
\label{sec:properties weaker than the FF}

In this section, we consider various generalizations of the FF property. First, we compare the three generalizations of the FF property mentioned in the previous section: the BF, IDF, and MCD-finite properties. Then we introduce the main factorization property that we investigate throughout the rest of the paper, the unrestricted finite factorization property.

\medskip
\subsection{Natural Generalizations of the FF Property}

The BF, IDF, and MCD-finite properties are all natural generalizations of the FF property. Indeed, for each property $\mathcal{P}$ weaker than the FF property in the diagram of Figure~\ref{fig:three weaker notions of the FF property}, we will provide some examples and construct a torsion-free monoid with the property $\mathcal{P}$ that does not have any of the other two properties; for the BF and IDF properties, we also construct integral domains. Let us start by taking $\mathcal{P}$ to be the IDF property.

\smallskip
\subsubsection{The IDF Property}

Every antimatter domain that is not a field is clearly an IDF domain that does not have the BF property. As the following proposition indicates, there exist IDF domains that have neither the BF property nor the MCD-finite property.

\begin{prop} \label{prop:IDF domain neither BFD nor MCD-finite}
    There exists an IDF domain that is neither a BFD nor an MCD-finite domain.
\end{prop}

\begin{proof}
    Let $R$ be an IDF domain such that $R[x]$ is not an IDF domain, which must exist by \cite{MO09}. We claim that $R$ does not have either the BF or the MCD-finite property. If $R$ had the BF property, then it would be an atomic IDF domain and so an FFD, whence $R[x]$ would also be an FFD because the FF property ascends to polynomial extensions. On the other hand, if $R$ had the MCD-finite property, then $R[x]$ would be an IDF domain (as it follows from~\cite{EK18} that the IDF ascends to polynomial extensions over the class of MCD-finite domains).
\end{proof}

Here we discuss a concrete construction of an IDF domain that has neither the BF property nor the MCD-finite property. 

\begin{example} \label{ex:a IDF that is neither BF nor MCD-finite}
    Let $x, y_1, y_2, y_3$ be pairwise distinct indeterminates, and let $(z_n)_{n \ge 1}$ be an infinite sequence of pairwise distinct indeterminates whose underlying set is disjoint from $\{ x, y_1, y_2, y_3 \}$. Now fix $p \in \pp$ and then consider the rank-$1$ valuation additive monoid $M := \mathbb N_0[1/p]$, whose Grothendieck group is $G := \zz[ 1/p]$. Consider the abelian group $A := G^4 \times G^{(\nn)}$, and let $\ff_p[X;A]$ denote the group algebra of $A$ with coefficients of $\ff_p$. Observe that the monomials of $\ff_p[X;A]$ have the form $X^s := x^{q_0} y_1^{q_1} y_2^{q_2} y_3^{q_3} \prod_{n \in \nn} z_n^{r_n}$ for some $s := ((q_0, q_1, q_2, q_3), (r_n)_{n \ge 1}) \in A$. We claim that the subring
    \[
        R := \mathbb F_p \left[x^q, z_n^q, \left(\frac{xy_1}{z_n}\right)^q, \left(\frac{xy_2}{z_n}\right)^q, \left(\frac{xy_3}{z_n}\right)^q \, : \, (n,q) \in \nn \times M \right]
    \]
    of $\ff_p[X;A]$ is an IDF domain that is neither a BFD nor an MCD-finite domain. First, observe that~$R$ contains no atoms because every element has a $p$-th root. As a result, $R$ is trivially an IDF domain that is not a BFD. Thus, we only need to argue that $R$ is not an MCD-finite domain. As the abelian group~$A$ is torsion-free, $\ff_p[X;A]$ is an integral domain and the units of $\ff_p[X;A]$ are precisely its nonzero monomials. As the $x$-valuation (or $y_1$-valuation, $y_2$-valuation, $y_3$-valuation) of any nonzero monomial of~$R$ is nonnegative, every unit $u$ in $R$ must be a monomial of the form
    \[
        u := \alpha \prod_{n \in \nn} z_n^{q_n} \in \ff_p[z_n^q : (n,q)\in\nn\times M],
    \]
    and so the fact that $M$ is a reduced monoid guarantees that $q_n = 0$ for every $n \in \nn$, hence $R^\times = \ff_p^\times$. Note that any two distinct elements in the set $\{z_n : n \in \nn \}$ are algebraically independent over $\ff_p$, and so non-associates in~$R$. Thus, in order to show that~$R$ is not an MCD-finite domain, it suffices to consider the nonempty finite subset
    \[
        S := \big\{ xy_1 + xy_3, \ xy_2+xy_3 \}
    \]
    of $R$ and argue that $z_n$ is an MCD of $S$ for every $n \in \nn$. To show this, we must demonstrate that, for any fixed index $n \in \nn$, the only common factors in~$R$ of the subset $S_n := \{a_n, b_n\}$ are those in $\ff_p^\times$, where
    \[
        a_n := x\frac{y_1}{z_n} + x\frac{y_3}{z_n} \quad \text{ and } \quad b_n := x\frac{y_2}{z_n} + x\frac{y_3}{z_n}.
    \]
    Suppose, by way of contradiction, that $S_n$ has a nonunit common divisor in~$R$ for some $n \in \nn$. Since $a_n$ and $b_n$ are homogeneous polynomials in $x$, any common factor of $S_n$ must be homogeneous in~$x$. Similarly, any common factor of $S_n$ must be homogeneous in~$z_n$. Furthermore, because $a_n$ contains no power of~$y_2$, the common factors of $S_n$ cannot contain any power of~$y_2$ and, analogously, the common factors of $S_n$ cannot contain any power of~$y_1$. Since $a_n - b_n$ contains no power of~$y_3$, the $y_3$-valuation of each term of any common factors of $S_n$ in $R$ is zero. 
    
    To complete the argument, let $d$ be a common divisor of $a_n$ and $b_n$. The preceding degree arguments show that $d$ involves none of the variables $y_1, y_2$, and $y_3$. In every monomial of $R$, if the exponents of $y_1,y_2$, and $y_3$ vanish then all the exponents of $x$ and $z_n$ (for every $n \in \nn$) are nonnegative. Since $a_n$ is homogeneous in each of $x,z_1,z_2,\ldots$, additivity of the maximum and minimum exponents in a product shows that $d$ is a monomial. Thus,
    \[
        d = \alpha x^r\prod_{k \in \nn} z_k^{q_k},
    \]
    for a coefficient $\alpha \in \ff_p^\times$ and exponents $r,q_k \in M$ (for every $k \in \nn$). Every monomial $t$ of $R$ satisfies
    \[
        \nu_x(t)\ge \sum_{j=1}^3\nu_{y_j}(t),\qquad
        \sum_i\max\{0,-\nu_{z_i}(t)\}\le \sum_{j=1}^3\nu_{y_j}(t).
    \]
    Applying the first inequality to the monomials of $a_n/d$ gives $r=0$. Applying the second gives $1+\sum_{k\ge1} q_k \le 1$, whence we obtain that $q_n=0$ for every $n \in \nn$. Thus, $d \in \ff_p^\times$. As a result, $z_n$ is an MCD of $S$ for every $n \in \nn$, and $R$ is not an MCD-finite domain. \hfill $\blacksquare$
\end{example}

	\medskip
	\subsubsection{The BF Property} 
	
	In the following example, we exhibit a one-dimensional monoid algebra over a field that is a BFD but is neither an IDF domain nor an MCD-finite domain. 
	
	\begin{example} \label{ex:Grams' monoid algebra is BF but neither IDF nor MCD-finite}
		Consider the submonoid $M := \{0\} \cup \qq_{\ge 1}$ of the additive monoid $\qq$. Observe that $M$ is a reduced rank-$1$ monoid with $\mathcal{A}(M) = [1,2) \cap \qq$. In addition, it follows from~\cite[Proposition~4.5]{fG19} that~$M$ has the BF property. To argue that $M$ does not have the IDF property, first observe that for each element $q \in M_{\ge 3}$ and $a \in \mathcal{A}(M)$, the inequality $q-a \ge 1$ holds and so $q-a \in M$ or, equivalently, $a \mid_M q$. Hence every element of the ideal $M_{\ge 3}$ is divisible by all the atoms. As $|\mathcal{A}(M)| = \infty$, we conclude that $M$ does not have the IDF property.
		
		Now we argue that $M$ does not have the MCD-finite property. To do so, fix $q,r \in M$ such that $2 < q < r < q+1$. Let us argue that the subset $S := \{q,r\}$ of $M$ has infinitely many MCDs. Observe that every element in the subset $D_S := [1, q-1] \cap \qq$ of $M$ is a common divisor of $S$: indeed, for each $d \in D_S$, we see that $q-d \ge q - (q-1) = 1$ and so $r-d > q-d \ge 1$, whence $d \mid_M q$ and $d \mid_M r$. Now fix $\epsilon \in \rr$ with $0 < \epsilon < \min\{1,q-2\}$. Because $1 < (q-1) - \epsilon < q-1$, it follows that the infinite subset $M_S := [(q-1)-\epsilon, q-1) \cap \qq$ of $M$ is contained in $D_S$, and so each $m \in M_S$ is a common divisor of $S$. On the other hand, for each $m\in M_S$, we have $1<q-m<2$, so $q-m$ is an atom. Its only nonzero divisor is itself, and it does not divide $r-m$ because $0<r-q<1$. Hence $q-m$ and $r-m$ have no nonzero common divisor, whence $m$ is an MCD of $S$. Hence every element in $M_S$ is an MCD of $S$, which implies that $S$ has infinitely many non-associate MCDs in $M$. As a consequence, $M$ does not have the MCD-finite property. 
		
		Finally, fix a field $F$ and consider the monoid algebra $F[M]$, which has dimension $1$ because $M$ has rank~$1$. Every nonunit of $F[M]$ has degree at least $1$, and degrees add under multiplication. Therefore every factorization of a nonzero element $f$ into nonunits has length at most $\lfloor\deg f\rfloor$, so $F[M]$ is a BFD. Now observe that $F^\times x^M$ is a divisor-closed submonoid of $F[M]^*$, and so the fact that the reduced monoid of $F^\times x^M$ is isomorphic to $M$ ensures that $F^\times x^M$ has neither the IDF property nor the MCD-finite property. Therefore the fact that $F^\times x^M$ is a divisor-closed submonoid of $F[M]^*$ implies that $F[M]$ has neither the IDF property nor the MCD-finite property. \hfill $\blacksquare$
	\end{example}

	\medskip
	\subsubsection{The MCD-finite Property} 
	
	Next we prove that Grams' monoid, used in the construction of the first known atomic domain that does not satisfy the ACCP, is an MCD-finite monoid that is neither a BFM nor an IDF monoid. %
	
	\begin{example} \label{ex:a PM that is MCD-finite but neither BF nor IDF}
		Let $(p_n)_{n \ge 1}$ be a strictly increasing sequence consisting of primes in $\pp_{\ge 3}$, and consider the submonoid
		\[
			M := \Big\langle \frac1{2^n p_n} : n \in \nn \Big\rangle.
		\]
		The monoid $M$ is often referred to as \emph{Grams' monoid}, as it first appeared in~\cite{aG74} as the monoid of exponents for a monoid algebra used to construct the first example of an atomic domain not satisfying the ACCP. In particular, it is well known that $M$ is atomic with
		\[
		\mathcal{A}(M) = \Big\{ \frac1{2^n p_n} : n \in \nn \Big\},
		\]
		but $M$ does not satisfy the ACCP. Hence $M$ is not a BFM. On the other hand, it is clear that $1 \in M$ is divisible by all the atoms, whence $M$ is not an IDF monoid. It was proved in~\cite[Theorem~4.7]{LWZ24} that $M$ has the MCD-finite property. \hfill $\blacksquare$
	\end{example}

	The examples we have exhibited earlier guarantee, in particular, that the (black) unbroken implication arrows shown in the diagram of~\ref{fig:three weaker notions of the FF property} are the only implication relation between the four properties considered so far in this section.
	\begin{center}
		\begin{figure}[ht]
			\begin{tikzcd} %
				\textbf{ BF }     \arrow[r, Leftarrow, shift right=0.8ex] \arrow[red, r, Rightarrow, "/"{anchor=center,sloped}, shift left=0.8ex]
				& \textbf{ FF }  \arrow[r, Rightarrow, shift right=0.8ex] \arrow[red,  r,  Leftarrow, "/"{anchor=center,sloped}, shift left=0.8ex] 
				\arrow[d, Rightarrow, shift right=0.8ex] \arrow[red, d, Leftarrow, "/"{anchor=center,sloped}, shift left=0.8ex] 
				& \textbf{ IDF } \\
				\phantom{ NA } & \  \  \textbf{ MCD-F } \  \ & \phantom{ NA}
			\end{tikzcd}
			\caption{The implications in the diagram show three properties implied by the FF property: the BF property, the IDF property, and the MCD-finite property. The (red) marked arrows emphasize that none of the shown implications are reversible.}
			\label{fig:three weaker notions of the FF property}
		\end{figure}
	\end{center}

	\medskip
	\subsection{The Unrestricted Finite Factorization Property}
	
	Let $R$ be an integral domain. Following Coykendall and Zafrullah~\cite{CZ04}, we say that $R$ is an \emph{unrestricted unique factorization domain} (U-UFD) or has the \emph{unrestricted unique factorization} (U-UF) property if $|\mathsf{Z}(r)| \le 1$ for all $r \in R^*$ (i.e., every atomic element of $R$ has a unique factorization). In the same paper, the authors proved that the class consisting of all U-UFDs strictly contains the class of all AP domains (i.e., integral domains where every atom is prime).
	\smallskip
	
	Motivated by the U-UF property, we proceed to introduce a weaker version of the FF property, which will play the central role in this paper.
	
	\begin{defn}
		We say that a monoid is an \emph{unrestricted finite factorization monoid} (U-FFM) or has the \emph{U-FF} property if every atomic element has only finitely many factorizations. We say that an integral domain is an \emph{unrestricted finite factorization domain} (U-FFD) or has the \emph{U-FF} property if its multiplicative monoid is a U-FFM.
	\end{defn}
	
	Note that every FFD that is not a UFD is an example of an integral domain that has the U-FF property but is not a U-UFD. It is clear that the U-FF property is a natural generalization of the FF and U-UF properties. This definition of a U-FFM may be rephrased in two equivalent intuitive ways, as indicated in the following proposition.
	
	\begin{prop}
		For a monoid $M$, the following conditions are equivalent.
		\begin{enumerate}
			\item[(a)] $M$ is a U-FFM.
			\smallskip
			
			\item[(b)] Every element of $M$ has finitely many factorizations.
			\smallskip
			
			\item[(c)] The atomic submonoid of $M$ is an FFM.
		\end{enumerate}
	\end{prop}
	
	As a U-FFM can be described as a monoid whose elements have only finitely many factorizations (possibly zero), we posit that the U-FF property could be even more natural and intuitive than the FF property.
	
	Next, we show that the U-FF property is not only a generalization of the FF property, but also a generalization of the IDF property.
	
	\begin{prop} \label{prop:IDF are U-FF}
		If a monoid has the IDF property, then it also has the U-FF property.
	\end{prop}
	
	\begin{proof}
		Given that any monoid has the IDF property (resp., U-FF property) if and only if its reduced monoid has the IDF property (resp., U-FF property), it is enough to restrict attention to the class consisting of all reduced monoids. With this in mind, let $M$ be a reduced IDF monoid, and let $A$ denote the atomic submonoid of $M$. Note that $A$ is an atomic IDF monoid and so an FFM. We are done once we argue that every element of $A$ has only finitely many factorizations in $M$: indeed, for each $a \in A$, the sets of factorizations $\mathsf{Z}_M(a)$ and $\mathsf{Z}_A(a)$ are the same, and from the fact that $A$ is an FFM we deduce that $|\mathsf{Z}_M(a)| < \infty$. Hence every atomic element of $M$ has finitely many factorizations, which means that~$M$ is a U-FFM.
	\end{proof}
	
	As the following example illustrates, the converse of Proposition~\ref{prop:IDF are U-FF} does not hold even within the class of rank-$1$ torsion-free monoids.
	
	\begin{example}
		Let $(p_n)_{n \ge 1}$ be the strictly increasing sequence whose underlying set is the set of primes, and consider the rank-one torsion-free monoid
		\[
			S := \bigg\langle \frac{p_n+1}{p_n^2} : n \in \nn \bigg\rangle.
		\]
		It is routine to show that $\mathcal{A}(S) = \big\{ a_n : n \in \nn \big\}$, where $a_n := \frac{p_n+1}{p_n^2}$ for every $n \in \nn$. Hence $S$ is atomic. Now fix a nonzero $q \in S$ and take $m \in \nn$ such that $m > \max\{q, \mathsf{d}(q) \}$, where $\mathsf{d}(q)$ is the positive denominator of $q$ in lowest terms. Let us show that $a_n \nmid_S q$ for any $n \in \mathbb N_{\ge m}$. Suppose the contrary by way of contradiction; then we may write
		\begin{equation} \label{eq:equality temp}
			q = \sum_{n=1}^\ell c_n \frac{p_n+1}{p_n^2}
		\end{equation}
		for some index $\ell \in \nn_{\ge m}$ and coefficients $c_1, \dots, c_\ell \in \nn_0$ with $c_\ell \ge 1$. Therefore, after applying the $p_\ell$-adic valuation to both sides of~\eqref{eq:equality temp}, we find that $p_\ell^2 \mid c_\ell$, which implies that
		\[
			p_\ell + 1 \le c_\ell \frac{p_\ell + 1}{p_\ell^2} = q - \sum_{n=1}^{\ell-1} c_n \frac{p_n + 1}{p_n^2} \le q.
		\]
		However, this is not possible because $p_\ell \ge p_m > m > q$, whence the atoms dividing $q$ in $S$ belong to the finite set $\{a_1, \dots, a_m \}$. Thus, $S$ is an IDF monoid (indeed, $S$ is an FFM because $S$ is atomic). Now, set
		\[
			M := S \cup \qq_{>1}, 
		\]
		and note that $M$ is a submonoid of $\qq_{\ge 0}$ because $S$ is a submonoid of $\qq_{\ge 0}$ and $\qq_{> 1}$ is an ideal of $\qq_{\ge 0}$. Every atom $a_n$ of $S$ is less than $1$, so any decomposition of $a_n$ in $M$ already lies in $S$. Thus, $\mathcal A(S)\subseteq\mathcal A(M)$. Conversely, a nonatom of $S$ remains reducible in $M$, and for every $r\in\qq_{>1}$ we can take $n$ large enough that $r-a_n>1$, yielding a decomposition into two nonzero elements of $M$. Therefore $\mathcal{A}(M) = \{a_n : n \in \nn \}$, and so the atomic monoid of $M$ is $\langle a_n : n \in \nn \rangle$, which is an FFM, and so $M$ has the U-FF property. However, $a_n \mid_M 2$ for all $n \in \nn$, whence $M$ does not have the IDF property. \hfill $\blacksquare$
	\end{example}

	The diagram in Figure~\ref{fig:four weaker notions of the FF property}, which is an extension of that shown in Figure~\ref{fig:three weaker notions of the FF property}, shows four classes of monoids defined by properties more general than the FF property. 
	\begin{center}
		\begin{figure}[ht]
			\begin{tikzcd} %
				\textbf{ BF }  \arrow[d, Leftarrow, shift right=1ex] \arrow[red, d, Rightarrow, "/"{anchor=center,sloped}, shift left=1ex]  & \phantom{ NA}  & \phantom{ NA } \\
				\textbf{ FF }     \arrow[r, Rightarrow, shift right=0.8ex] \arrow[red, r, Leftarrow, "/"{anchor=center,sloped}, shift left=0.8ex] 
				\arrow[d, Rightarrow, shift right=1ex] \arrow[red, d, Leftarrow, "/"{anchor=center,sloped}, shift left=1ex] 
				& \textbf{ IDF }    \arrow[r, Rightarrow, shift right=0.8ex] \arrow[red, r, Leftarrow, "/"{anchor=center,sloped}, shift left=0.8ex] 
				& \textbf{ U-FF } \\
				\ \ \textbf{  MCD-F} \  \ & \phantom{ NA } \ & \phantom{ NA }
			\end{tikzcd}
			\caption{The diagram shows all the implications among the discussed properties generalizing the FF property. The (red) marked arrows emphasize that none of the shown implications are reversible.}
			\label{fig:four weaker notions of the FF property}
		\end{figure}
	\end{center}
	
	In light of Proposition~\ref{prop:IDF domain neither BFD nor MCD-finite}, Example~\ref{ex:Grams' monoid algebra is BF but neither IDF nor MCD-finite}, and Example~\ref{ex:a PM that is MCD-finite but neither BF nor IDF}, we can conclude that the diagram in Figure~\ref{fig:four weaker notions of the FF property} shows all possible implication arrows.
	\medskip
	
	We conclude this section constructing, for each $\ell \in \nn \cup \{\infty\}$, a non-atomic IDF domain containing exactly~$\ell$ irreducibles up to associates.
	
	\begin{prop} \label{prop:IDF domain with prescribed number of atoms}
		For any $\ell \in \nn \cup \{\infty \}$, there exists a non-atomic IDF domain containing exactly~$\ell$ irreducibles up to associates.
	\end{prop}
	
	\begin{proof}
		First, suppose that $\ell \in \nn$. Let $M$ be the additive monoid consisting of all nonnegative dyadic rationals, and let $R_0$ denote the monoid algebra of $M$ in an indeterminate $x_0$ over a field $\ff$. Since $M$ is cancellative and torsion-free, $R_0$ is an integral domain. Observe that the kernel $I$ of the ring homomorphism $R_0 \to \ \ff$ given by the assignment $f \mapsto f(0)$ is a maximal ideal of $R_0$. Let $x_1, \dots, x_\ell$ be distinct indeterminates (i.e., algebraically independent) over the ring $R_0$, and consider the polynomial ring
		\[
			R := R_0[x_1, \dots, x_\ell].
		\] 
		As $I$ is a prime ideal of $R_0$, its extension $I_0 := IR$ to $R$ is a prime ideal in $R$. On the other hand, it is clear that the ideal $I_k := x_k R$ is a prime ideal of $R$ for every $k \in \ldb 1, \ell \rdb$ and, therefore, the set
		\[
			S := R \setminus \bigcup_{k=0}^\ell I_k
		\]
		is a multiplicative subset of $R$. In addition, $S$ is saturated in $R$, which means that for all $r_1, r_2 \in R$ with $r_1 r_2 \in S$ we can ensure that $r_1,r_2 \in S$. To argue this, assume that $r_1 \notin S$ and then take $j \in \ldb 0,\ell \rdb$ so that $r_1 \in I_j$, whence $r_1 r_2 \in I_j \subseteq \bigcup_{k=0}^\ell I_k \subseteq R \setminus S$. Therefore, after localizing $R$ at $S$, we obtain an integral domain $T$ with $T^\times = \{s/t : s,t \in S\}$ and $T^\times \cap R = S$. Before proceeding, let us argue the following claim.
		\smallskip
		
		\noindent \textsc{Claim.} $\mathcal{A}(T) = \bigcup_{k=1}^\ell T^\times x_k$.
		\smallskip
		
		\noindent \textsc{Proof of Claim.} For the inclusion $\bigcup_{k=1}^\ell T^\times x_k \subseteq \mathcal{A}(T)$, it suffices to argue that $x_1, \dots, x_\ell$ are irreducibles in $T$. To do so, fix $k \in \ldb 1,\ell \rdb$. First, observe that $x_k \notin T^\times$ as, otherwise, we could take $u, v \in S$ such that $v = x_k u \in x_k R = I_k$, which is not possible because $I_k$ is disjoint from~$S$. Thus, $x_k \notin T^\times$. Now write $x_k = (r_1/s_1)(r_2/s_2)$ for some $r_1, r_2 \in R \setminus S$ and $s_1, s_2 \in S$. Since $r_1 r_2 = x_k(s_1 s_2) \in I_k$, the fact that $I_k$ is a prime ideal of $R$ ensures that either $r_1 \in x_kR$ or $r_2 \in x_kR$. Assume, without loss of generality, that $r_1 \in x_kR$ and take $r_3 \in R$ such that $r_1 = x_k r_3$. Using this, along with the equality $r_1 r_2 = x_k s_1 s_2$, we obtain
		\[
			x_k s_1 s_2 = r_1 r_2 = x_k r_3 r_2.
		\]
		Observe that if $r_1$ and $r_2$ both belonged to $R \setminus S$ then $r_2 r_3 = s_1 s_2 \in S$, which implies that $r_2 \in S$ as $S$ is saturated. Therefore $r_2/s_2 \in T^\times$, and so $x_k$ is irreducible in $T$.
		
		For the inclusion $\mathcal{A}(T)\subseteq\bigcup_{k=1}^\ell T^\times x_k$, take a nonzero nonunit $f \in T$ such that $f \notin \bigcup_{i=1}^\ell T^\times x_i$, and let us verify that $f$ is not irreducible in $T$. After multiplying $f$ by a suitable element of $S$, we can assume that $f \in R$, whence $f \in R \setminus S = \bigcup_{k=0}^\ell I_k$ because $f$ is not a unit in~$T$. We consider the following two cases.
		\smallskip
		
		\textsc{Case 1:} $f \in \bigcup_{k=1}^\ell I_k$. In this case, we can take $k \in \ldb 1,\ell \rdb$ such that $f \in I_k = x_kR \subset x_kT$, and so we can write $f = x_k r_k$ for some $r_k \in R$. Note that $r_k \notin T^\times$ because $f$ and $x_k$ are not associate in $T$. Therefore $f \notin \mathcal{A}(T)$.
		\smallskip
		
		\textsc{Case 2:} $f \in I_0$. In this case, we can write $f = f_1 r_1 + \dots + f_m r_m$ for some nonzero $f_1, \dots, f_m \in I$ and $r_1, \dots, r_m \in R$. For each $i\in\ldb1,m\rdb$, write $f_i=x_0^{q_i}g_i$ with $q_i\in M_{>0}$ and $g_i\in R_0$. Choose $q\in M$ with $0<q<\min_iq_i$. Then $f=x_0^q h$ with $h\in I_0$. Both $x_0^q$ and $h$ lie in $I_0T$ and are nonunits of $T$, so $f$ is reducible.
		\smallskip
		
		As a consequence, the only atoms of $T$ are precisely those in the set $\bigcup_{k=1}^\ell T^\times x_k$ and, therefore, the claim is established.
		\smallskip
		
		We finally argue that $T$ is not atomic: indeed, the element $x_0^m$ is not atomic in $T$ for any nonzero $m \in M$. To see this, fix $m \in M$ with $m > 0$. Observe that $x_0^m \in IR = I_0$ and so $x_0^m \notin S=T^\times\cap R$. Hence $x_0^m$ is a nonunit of $T$. Notice now that if, for some $k \in \ldb 1, \ell \rdb$, we could write $x_0^m = x_k t_k$ for some $t_k \in T$, then $s_k x_0^m \in x_kR = I_k$ for some $s_k \in S$, and so $I_k$ would be a prime ideal of $R$ containing $s_k x_0^m$ but neither $s_k$ nor $x_0^m$. Hence $x_0^m$ is a nonunit of $T$ that is not divisible by any of the irreducibles of $T$, and so we conclude that $x_0^m$ is a nonzero element of $T$ that is not atomic.
		\smallskip
		
		The prime elements $x_1,\dots,x_\ell$ are pairwise non-associate in $T$. Indeed, if $x_i \mid_T x_j$ with $i \ne j$ then $s x_j \in x_iR$ for some $s \in S$, contrary to the primality of $x_iR$ and its disjointness from $S$. Thus, $T$ has precisely $\ell$ irreducibles up to associates and is IDF.
		
		Let us produce now an IDF domain with countably many irreducibles up to associates. For this, set
		\[
			R := \ff[M \times \nn_0^{(\nn)}] = R_0[x_1, x_2, \ldots],
		\]
		and then consider the sequence of ideals $(I_k)_{k \ge 0}$ of $R$ defined as follows: $I_0 :=IR$ and $I_k := x_kR$ for every $k \in \nn$. Now let $T$ denote the localization of $R$ at its multiplicative subset $S := R \setminus \bigcup_{k \ge 0}I_k$. The same proof gives
		\[
			\mathcal A(T) = \bigcup_{k \ge 1} T^\times x_k,
		\]
		where $T^\times x_i$ and $T^\times x_j$ are distinct for any distinct indices $i,j \in \nn$, and also shows that $x_0^q$ is not atomic for any $q \in M_{>0}$. To verify that $T$ satisfies the IDF property, write a nonzero element of $T$ as $f/s$ with $f \in R$ and $s \in S$. Choose $\ell \in \nn$ such that $f \in R_0[x_1, \dots, x_\ell]$. If $x_k \mid_T f$ then $tf \in x_kR$ for some $t \in S$. The primality of $x_kR$, along with the fact that $t \notin x_kR$, implies that $f \in x_kR$, which is impossible if $k > \ell$. Thus, only $x_1, \dots, x_\ell$ can divide $f$ up to associates in $T$. As a consequence, $T$ is an IDF domain.
	\end{proof}

	\bigskip
	\section{Localization}
	\label{sec:localization}
	
	In this section, we study the ascent and descent of the U-FF property with respect to localization, following the approach used by Anderson, Anderson, and Zafrullah in their study of the BF and FF properties~\cite{AAZ92}. Throughout this section, $R$ denotes an integral domain and $S$ denotes a saturated multiplicative subset of $R$.

	\medskip
	\subsection{Splitting Multiplicative Set and Inert Extensions}
	
	We say that $a \in R$ satisfies the \emph{splitting condition} with respect to $S$ if the following holds:
	\begin{equation} \label{eq:EC=Id}
		aR \cap sR = asR \ \ \text{for all} \ \ s \in S.
	\end{equation}
	Observe that $a \in R$ satisfies~\eqref{eq:EC=Id} if and only if
	\[
	((aR)^e)^c = aR_S \cap R = \bigcup_{s \in S} \frac{a}sR \cap R = aR
	\]
	(this is \cite[Lemma~1.2]{AAZ92}). Hence the elements of $R$ that satisfy the splitting condition are those whose principal ideals are $S$-saturated. Set
	\begin{equation} \label{eq:P(S)}
		P_S(R) := \{a \in R^* : aR \text{ is }S\text{-saturated} \}.
	\end{equation}
	It turns out that $P_S(R)$ is an intermediate monoid of the extension $R^\times \subseteq R^*$.
	
	\begin{lem}
		Let $R$ be an integral domain, and let $S$ be a multiplicative set of $R$. Then $P_S(R)$ is an intermediate monoid of the extension $R^\times \subseteq R^*$.
	\end{lem}
	
	\begin{proof}
		Fix $a,b \in P_S(R)$, and let us show that $ab \in P_S(R)$. To do so, take $(r,s) \in R \times S$ with $rs \in abR$. Then $sr = aby$ for some $y \in R$. As $aR$ is $S$-saturated, $r \in aR$ and so we can write $r = az$ for some $z \in R$.  Thus, $saz = sr = aby$ and so $sz = by \in bR$. This, along with the fact that $bR$ is $S$-saturated, ensures that $z \in bR$. Therefore $sr = s(az) \in sabR$, which implies that $r \in abR$. Thus, $abR$ is $S$-saturated and we conclude that $ab \in P_S(R)$. The inclusion $R^\times \subseteq P_S(R)$ is an immediate consequence of the fact that $R$ is an $S$-saturated ideal, that is, $(R^e)^c = R_S^c = R$.
	\end{proof}

	The set consisting of all extensions of $S$-saturated principal ideals to $R_S$ is a submonoid of $\mathcal{I}(R_S)$ and plays a relevant role in this section:
	\[
		\mathcal{H}_S(R) := \{aR_S : a \in P_S(R)\}.
	\]
	Next, we argue that the ideal contraction map with respect to the localization homomorphism $R\to R_S$, restricted to $\mathcal H_S(R)$, is a monoid homomorphism.
	
	\begin{lem}
		Let $R$ be an integral domain, and let $S$ be a multiplicative set of $R$. Then the following statements hold.
		\begin{enumerate}
			\item $\mathcal{H}_S(R)$ is a submonoid of $\mathcal{I}(R_S)$.
			\smallskip
			
			\item $(\cdot)^c \colon \mathcal{H}_S(R) \to \mathcal{I}(R)$ is a monoid homomorphism with $\mathcal{H}_S(R)^c = \{aR : a \in P_S(R) \}$.
		\end{enumerate}
	\end{lem}
	
	\begin{proof}
		(1) For any elements $a,b \in P_S(R)$, the previous lemma ensures that $ab \in P_S(R)$ and, as a consequence, $(aR_S)(bR_S) = abR_S \in \mathcal{H}_S(R)$. The fact that $R_S$ is the identity element of $\mathcal{H}_S(R)$ follows immediately as $1 \in P_S(R)$.
		\smallskip
		
		(2) It is clear that $\mathcal{H}_S(R)^c = \{aR : a \in P_S(R) \}$ because $(aR_S)^c = aR_S \cap R = aR$ for all $a \in P_S(R)$. Also, as $(R_S)^c = R_S \cap R = R$, the identity elements are preserved by $(\cdot)^c$. To argue that the restriction of $(\cdot)^c$ to $\mathcal{H}_S(R)$ is multiplicative, it suffices to fix $a,b \in P_S(R)$ and use the equalities $(aR_S)^c = aR$ and $(bR_S)^c = bR$ to check that
		\[
			(aR_S \cdot bR_S)^c = abR_S \cap R = abR = aR \cdot bR = (aR_S)^c \cdot (bR_S)^c, 
		\]
		where the second equality holds because $ab \in P_S(R)$. Hence the ideal contraction map of the localization homomorphism $R \to R_S$ restricted to $\mathcal{H}_S(R)$ is a monoid homomorphism.
	\end{proof}
	
	Unlike the ideal extension map, the ideal contraction map of a localization homomorphism may not be a monoid homomorphism. The following example illustrates this.
	
	\begin{example}
		Let $k$ be a field, and consider the integral domain $R := k[x,y,z]/(xy - z^2)$. Let $S := R^\times\{x^n : n \in \mathbb{N}_0\}$. Consider the ideals of $R_S$:
		\[
			J_1 := zR_S \quad \text{ and } \quad J_2 := zR_S
		\]
		First, let us compute their contractions. We see that $J_1^c = zR_S \cap R$. Notice that in $R_S$, the equality $y = z^2/x = z(z/x)$ holds, which means that $y \in zR_S$ and so $y \in J_1^c$. Thus, the prime ideal $(y, z)$ is contained in $J_1^c$. Because $x \notin (y, z)$, no further elements are brought in by the $x$-saturation, yielding $J_1^c = (y, z)$. The same holds for $J_2$, and so $J_2^c = (y, z)$. After multiplying the contractions, one obtains
		\[
			J_1^c \cdot J_2^c = (y, z)^2 = (y^2, yz, z^2) = (y^2, yz, xy) = y(x, y, z).
		\]
		On the other hand, let us compute the contraction of their product. In $R_S$, we see that $J_1 J_2 = z^2 R_S = xy R_S = y R_S$ because $x$ is a unit in $R_S$. The contraction is $(J_1 J_2)^c = yR_S \cap R$. Since the ideal $(y)$ is saturated with respect to $x$ in $R$, we obtain $(J_1 J_2)^c = yR$. Thus,
		\[
			J_1^c \cdot J_2^c = y(x,y,z) \subsetneq yR = (J_1 J_2)^c.
		\]
		Therefore we conclude that the ideal contraction map is not a monoid homomorphism. \hfill $\blacksquare$
	\end{example}
	
	A \emph{splitting decomposition} of $r \in R^*$ with respect to $S$ is an expression $r = a_rs_r$ with $(a_r, s_r) \in P_S(R) \times S$, in which case we say that $a_r$ is a \emph{splitting representative} of $r$ with respect to $S$.
	
	\begin{defn}
		Let $R$ be an integral domain and let $S$ be a saturated multiplicative set of $R$. Then~$S$ is called \emph{splitting} if $R^* = P_S(R)S$.
	\end{defn}
	
	If $as = bt$ for some $a,b \in P_S(R)$ and $s,t \in S$, then the equality $aR_S = bR_S$ holds and this implies that $aR = bR$ or, equivalently, that $a$ and $b$ are associate in $R$. Thus, splitting representatives are unique up to associates.

	\medskip
	\subsection{Ascent of the U-FF Property}

	The first main result we will prove in this section is that a sufficient condition for the ascent of the U-FF property from $R$ to $R_S$ is that $S$ is a splitting multiplicative set. First, we exhibit an example illustrating that, in general, non-atomic elements can suddenly become atomic upon localization and possess infinitely many factorizations.
	
	\begin{example} \label{ex:uff_localization_fails}
		For a prime $p \in \pp$, consider the subring $R := \mathbb{Z}_{(p)} + x\mathbb{R}\ldb x \rdb$ of the ring of power series $\rr\ldb x \rdb$. Notice that $pR$ is a prime ideal of $R$ because $R/pR \cong \ff_p$, whence $p$ is a prime element of $R$. Let $S$ be the multiplicative set of $R$ generated by $p$, that is,
		\[
			S := R^\times \{ p^n : n \in \nn_0 \}. 
		\]
		We first argue that $R$ is a U-FFD. Since $\rr\ldb x\rdb$ is a local domain and $\zz_{(p)}$ has a finite nonempty set of irreducibles up to associates, it follows from \cite[Proposition 4.3]{AAZ92} that $R$ is an IDF domain. Therefore $R$ is a U-FFD by Proposition~\ref{prop:IDF are U-FF}.
		
		Let us argue now that $R_S$ is not a U-FFD. Observe that
		\[
			R_S = R\Big[\frac1p\Big] = \mathbb{Q} + x\mathbb{R}\ldb x \rdb.
		\]
		The units of $R_S$ are precisely $\qq^\times + x\rr \ldb x \rdb$. Every nonunit has positive $x$-order, so the monomial $ux$ is irreducible in $R_S$ for all $u \in \rr^\times$. Therefore $z_u := ux \cdot u^{-1}x$ is a length-$2$ factorization of $x^2$ for all $u \in \rr^\times$. As $u_1 x, u_2 x \in \mathcal{A}(R_S)$ are associate in $R_S$ if and only if $u_1/u_2 \in R_S^\times \cap \mathbb{R} = \mathbb{Q}^\times$, the fact that $\mathbb{R}^\times / \mathbb{Q}^\times$ is an infinite set guarantees that $\{z_u : u \in \rr^\times \}$ is an infinite subset of $\mathsf{Z}(x^2)$: each factorization identifies at most the two cosets represented by $u$ and $u^{-1}$. Hence $x^2$ is an atomic element with infinitely many factorizations, and so we conclude that $R_S$ is not a U-FFD. \hfill $\blacksquare$
	\end{example}
	
	As part of~\cite{AAZ92}, the authors proved that if $S$ is a splitting multiplicative set of $R$, then for any nonzero $r \in R$ with splitting decomposition $(a_r,s_r)$, the element $r$ is an irreducible (resp., a prime) in $R_S$ if and only if $a_r$ is irreducible (resp., a prime) in $R$.
	Taking this into account, we now prove that the U-FF property ascends to localizations at splitting multiplicative sets.
	
	\begin{theorem} \label{thm:uff_ascent}
		Let $R$ be an integral domain, and let $S$ be a splitting multiplicative set of $R$. If $R$ is a U-FFD, then $R_S$ is a U-FFD.
	\end{theorem}
	
	\begin{proof}
		Assume that $R$ is a U-FFD. To prove that the localization $R_S$ is also a U-FFD, fix an atomic element $q \in R_S$ (we can assume that $q \in R$), write $q = q_1 \cdots q_m$ for some $q_1, \dots, q_m \in \mathcal{A}(R_S)$, and let us prove that $q$ has only finitely many factorizations in $R_S$. 
		
		For each $i \in \ldb 1,m \rdb$, there exists $u_i \in R_S^\times$ such that $u_iq_i \in R$. After replacing each $q_i$ by an associate in $R_S$, we can assume that $q_1, \dots, q_m \in R$. Since $S$ is a splitting multiplicative set, each $q_i$ has a splitting representative $a_i \in P_S(R)$. By \cite[Corollary 1.4(c)]{AAZ92}, each $a_i$ is irreducible in $R$. By setting $a := a_1 \cdots a_m$ and using the previous lemma, we obtain that $a \in P_S(R)$. Moreover, $q$ and $a$ are associates in $R_S$. After multiplying $q$ by a unit of $R_S$ if necessary, we may therefore suppose that $q = as$ for some $s \in S$. Since $a$ is atomic in~$R$, the fact that $R$ is a U-FFD ensures that $a$ has only finitely many factorizations in~$R$.
		
		Now let $q = \prod_{j=1}^n z_j$ be any factorization of $q$ in $R_S$. Proceeding as above, each $z_j$ is associate in $R_S$ to a splitting representative $b_j \in P_S(R)$, and each $b_j$ is irreducible in $R$. Set $b := b_1 \cdots b_n$. Then $b \in P_S(R)$ and $b$ is associate to $q$, and therefore to $a$, in $R_S$. Since both $aR$ and $bR$ are $S$-saturated, we obtain $aR = bR$. Hence $b_1 \cdots b_n$ is a factorization in $R$ of an associate of $a$. As $a$ has only finitely many factorizations in $R$, only finitely many factorizations of $q$ can occur in $R_S$, up to order and associates. Hence $q$ has only finitely many factorizations in $R_S$, and so $R_S$ is a U-FFD.
	\end{proof}

	\medskip
	\subsection{Nagata-Type Theorem for U-FFDs}
	
	Recall that an extension $A \subseteq B$ of integral domains is inert if, for all $b_1, b_2 \in B$ such that $b_1b_2 \in A \setminus \{0\}$, there exists $u \in B^\times$ such that $ub_1, u^{-1}b_2 \in A$. In this case, every atom of $A$ is either an atom or a unit of $B$. In general, the extension $R \subseteq R_S$ obtained by localization is not inert. However, $R \subseteq R_S$ is inert provided that $S$ is generated by primes~\cite[Proposition~1.9]{AAZ92}. In addition, $R \subseteq R_S$ is an inert extension if $S$ is splitting~\cite[Lemma~1.2]{AAZ92}.
	
	Nagata's classic theorem states that if $R$ is atomic, $S$ is generated by primes, and $R_S$ is a UFD, then $R$ is a UFD. Anderson, Anderson, and Zafrullah generalized this descent to FFDs when $S$ is a splitting multiplicative set generated by primes. We show the descent holds for U-FFDs as well.
	
	\begin{theorem} \label{thm:uff_descent}
		Let $R$ be an integral domain and $S$ a splitting multiplicative set generated by primes. If $R_S$ is a U-FFD then $R$ is a U-FFD.
	\end{theorem}
	
	\begin{proof}
		After replacing $S$ by its saturation (i.e., divisor closure in $R$), we can assume that $S$ is divisor-closed and we do so. Let $P$ be a maximal set of non-associate primes in $S$. 
		
		Take $r \in R$ to be an atomic element, and let us show that $r$ has only finitely many factorizations in~$R$. Let $(a,s) \in R \times S$ be a splitting decomposition of $r$. Since $S$ is generated by prime elements of $R$ up to units, $s$ has a unique prime factorization. Moreover, no prime in $P$ divides $a$: if $a=pb$ with $p\in P$, then $pb\in aR$ and the $S$-saturation of $aR$ would imply $b\in aR$, forcing $p$ to be a unit. Therefore every factorization of $r=as$ contains precisely the prime factors of $s$ from $P$, up to associates. Removing those factors shows that $a$ is atomic and that its irreducible factors are not associate to any primes in $P$.
		
		Because $S$ is generated by primes, the extension $R \subseteq R_S$ is inert~\cite[Proposition~1.9]{AAZ92}. Therefore every irreducible factor of $a$ in $R$ either remains an atom or becomes a unit in $R_S$ by~\cite[Lemma~1.1]{AAZ92}. Since no irreducible factor of $a$ is associate to a prime in $S$, none of its irreducible factors becomes a unit in $R_S$. Consequently, $a$ is atomic in $R_S$.
		
		Since $R_S$ is a U-FFD, $a$ has only finitely many factorizations in $R_S$. By \cite[Corollary~1.4(a)]{AAZ92}, any two factorizations of $a$ in $R$ that become associate in $R_S$ are already associate in $R$. Thus, the number of factorizations of $a$ in $R$ is bounded above by the number of factorizations of $a$ in $R_S$. Combining this with the unique prime factorization of $s$, we conclude that $r$ has only finitely many factorizations in~$R$. Hence $R$ is a U-FFD.
	\end{proof}

	\bigskip
	\section{The D+M Construction}
	\label{sec:D+M}
	
	Let $T$ be an integral domain, and let $K$ and $M$ be a subfield of $T$ and a nonzero maximal ideal of~$T$, respectively, such that $T=K+M$. For a subring $D$ of $K$, set
	\begin{equation} \label{eq:D+M construction}
		R := D+M.
	\end{equation}
	The decomposition of $R$ as an internal sum of $D$ and $M$ in $T$ is referred to as the $D+M$ \emph{construction} and has been largely studied in commutative ring theory. Observe that the underlying additive group of $T$ is the direct sum of $K$ and $M$, and so every element $z \in T$ can be written uniquely as $z = a + m$ for some $a \in K$ and $m \in M$, in which case $z \in R$ if and only if $a \in D$. From this we infer that 
	\[
		R^\times = \{ d+m \in T^\times : d \in D^\times, \ m \in M\};
	\] 
	in particular, $R^\times=T^\times\cap R$ when $D$ is a field.
	
	Before studying the behavior of the U-FF property under the $D+M$ construction, let us argue the following lemma.
	
	\begin{lem}
		Let $T$ be an integral domain, and let $K$ and $M$ be a subfield of $T$ and a nonzero maximal ideal of $T$, respectively, such that $T = K+M$. For a subring $D$ of $K$, set $R=D+M$. Then the extension $R \subseteq T$ is inert.
	\end{lem}
	
	\begin{proof}	
		Take $x, y \in T^*$ such that $xy \in R$, and write $x = a + m_1$ and $y = b + m_2$ for some $a, b \in K$ and $m_1, m_2 \in M$. We want to pick $u \in T^\times$ such that $ux,u^{-1}y \in R$. First observe that
		\[
			xy = (a + m_1)(b + m_2) = ab + (am_2 + bm_1 + m_1m_2),
		\]
		and so the fact that $am_2 + bm_1 + m_1m_2 \in M$ implies that $ab \in D$. If $a = b = 0$, then $x = m_1 \in R$ and $y = m_2 \in R$, and we can simply take $u = 1$. Therefore we assume that either $a$ or $b$ is nonzero: assume that $a \neq 0$. In this case, $a \in K^\times \subseteq T^\times$ and so we set $u := a^{-1} \in K^\times \subseteq T^\times$. Observe that $ux = a^{-1}(a + m_1) = 1 + a^{-1}m_1 \in D + M$ and $u^{-1}y = ab + am_2 \in D+M$.
	\end{proof}
	
	As an immediate consequence we obtain the following corollary.
	
	\begin{cor}
		Let $T$ be an integral domain, and let $K$ and $M$ be a subfield of $T$ and a nonzero maximal ideal of $T$, respectively, such that $T = K+M$. For a subring $D$ of $K$, set $R=D+M$. Then $\mathcal{A}(R) \subseteq \mathcal{A}(T) \cup T^\times$.
	\end{cor}

	Let us take a look at a couple of examples in connection to the U-FF property. The first one, which is the one in Example~\ref{ex:uff_localization_fails}, illustrates that the U-FF property does not descend from~$T$ to~$R$ under the $D+M$ construction.
	
	\begin{example}
		Consider the integral domain $T := \rr\ldb x \rdb$, which is a UFD and, therefore, a U-FFD. We can write $T = K+M$, where $K$ is the field $\rr$ and $M$ is the maximal ideal $x\rr\ldb x\rdb$. After taking $D := \qq$, we obtain that $D+M := \qq + x\rr\ldb x \rdb$. Even though $T$ is a U-FFD, we have already argued in Example~\ref{ex:uff_localization_fails} that $D+M$ is not a U-FFD. \hfill $\blacksquare$
	\end{example}
	
	The domain $R:=\qq+x\rr\ldb x\rdb$ in the previous example is neither IDF nor MCD-finite. Indeed, $x^2$ has infinitely many non-associate irreducible divisors $ux$, and these are also MCDs of $\{x^2,\sqrt2 x^2\}$. Indeed, a common nonunit divisor of $x/u$ and $\sqrt2 x/u$ would have $x$-order $1$, and the two quotients would be units with rational nonzero constant terms. Their ratio is $\sqrt2$, a contradiction.
	
	When an integral domain $R$ contains atoms, we will provide a suitable theorem describing the behavior of the U-FF property with respect to the $D+M$ construction.
	
	\begin{lem} \label{lem:auxiliary results in the D+M construction}
		Let $T$ be an integral domain, and let $K$ and $M$ be a subfield of $T$ and a nonzero maximal ideal of $T$, respectively, such that $T=K+M$. For a subfield $k$ of $K$, set $R=k+M$. Then the following conditions hold.
		\begin{enumerate}
			\item If $T$ is local then $T^\times = K^\times(1+M)$ and $R^\times = k^\times(1+M)$.
			\smallskip
			
			\item $A(T) \cap R \subseteq A(R)$.
		\end{enumerate}
	\end{lem}
	\begin{proof}
		(1) If $T$ is local, its units are precisely the elements with nonzero residue in $K$. As a consequence, $T^\times=K^\times(1+M)$. Also, we can see that 
		\[
			(1+m)^{-1}=1-m(1+m)^{-1}\in1+M
		\]
		for every $m \in M$, which implies that $R^\times = k^\times(1+M)$.
		\smallskip
		
		(2) Since $k$ is a field, $R^\times = T^\times\cap R$. Inertness therefore gives $\mathcal A(R) = \mathcal A(T) \cap R$. Let $r \in A(T) \cap R$. If $r$ is a unit of $T$ then it is a unit of $R$. Otherwise, write $r = u p_1 \cdots p_n$, where $u \in T^\times$ and $p_i \in \mathcal A(T)$ for every $i \in \ldb 1, n \rdb$. Multiply each $p_i$ that does not belong to $M$ by the inverse of its residue in $K$, leaving the factors in $M$ unchanged, to obtain $q_i \in \mathcal A(R)$ and $r = vq_1 \cdots q_n$ with $v \in T^\times$. If some $q_i$ belongs to $M$, absorb $v$ into one such factor; it remains in $\mathcal A(T) \cap M \subseteq \mathcal A(R)$. Otherwise, every $q_i$ has residue $1$ and $v$ has the same nonzero residue as $r$, which lies in $k^\times$, so $v \in R^\times$. In either case $r \in A(R)$, and this concludes the proof.
	\end{proof}
	
	We are in a position to prove the main result of this section.
	
	\begin{theorem} \label{thm:UFF-DplusM}
		Let $T$ be an integral domain, and let $K$ and $M$ be a subfield of $T$ and a nonzero maximal ideal of $T$, respectively, such that $T = K+M$. For a subfield $k$ of $K$, set $R := k+M$. Then the following statements hold.
		\begin{enumerate}
			\item If $M$ contains an irreducible element of $T$ then $R$ is a U-FFD if and only if $T$ is a U-FFD and the group $K^\times/k^\times$ is finite.
			\smallskip
			
			\item If $M$ contains no irreducible elements of $T$ then $R$ is a U-FFD if and only if $T$ is a U-FFD.
		\end{enumerate}
	\end{theorem}
	
	\begin{proof}
		Consider the residue map $\varphi \colon T \to K$, which is the ring homomorphism that maps a given element of $T$ to its projection on $K$ (note that $T$ is the direct sum of $K$ and $M$). After setting $U := (1+M) \cap T^\times$, we obtain that
		\[
			T^\times = K^\times U \qquad \text{and} \qquad R^\times = k^\times U. 
		\]
		In addition, observe that  $T^\times/R^\times \cong K^\times/k^\times$. Indeed, divide a unit of $T$ by its nonzero residue to obtain an element of $U$. The same argument applies to $R$. By the preceding lemma and its proof, $\mathcal A(R) = \mathcal A(T)\cap R$, and every $T$-factorization of an element of $R$ lifts to an $R$-factorization after adjusting its factors by units of $T$.
		
		Assume first that $R$ is a U-FFD. For any atomic $a\in T$, either $a \in M \subseteq R$, or $\varphi(a)^{-1}a\in R$. Thus, some $T$-associate $r$ of $a$ belongs to $R$ and is atomic there. Every $T$-factorization of $r$ lifts to an $R$-factorization, so $|\mathsf Z_T(a)| = |\mathsf Z_T(r)| \le |\mathsf Z_R(r)| < \infty$. Hence $T$ is a U-FFD.
		\smallskip
		
		(1) Suppose that $p \in \mathcal A(T) \cap M$. For each $c \in K^\times$, both $cp$ and $c^{-1}p$ belong to $\mathcal A(R)$, and
		\[
			p^2 = (cp)(c^{-1}p).
		\]
		Moreover, $cp\sim_R dp$ if and only if $c/d \in k^\times$. Each unordered factorization above identifies at most two cosets of $k^\times$ in $K^\times$. Thus, if $R$ is a U-FFD, the quotient $K^\times/k^\times$ is finite.
		
		Conversely, assume that $T$ is a U-FFD and $K^\times/k^\times$ is finite. Fix an atomic element $r\in R$. Its $R$-factorizations give $T$-factorizations, of which there are finitely many. A $T$-associate class of an atom contains at most $|T^\times/R^\times|$ distinct $R$-associate classes of atoms. Therefore each fixed $T$-factorization of $r$ has only finitely many lifts up to order and $R$-associates. It follows that $r$ has finitely many $R$-factorizations.
		\smallskip
		
		(2) Suppose that $M\cap\mathcal A(T) = \varnothing$ and that $T$ is a U-FFD. Every atomic element $r \in R$ lies outside~$M$, since $M$ is prime and no atom lies in $M$. If $a,b \in R \setminus M$ are $T$-associates, say $a=ub$, then $\varphi(u) = \varphi(a)/\varphi(b) \in k^\times$, so $u \in R^\times$. Thus, two $R$-factorizations of $r$ giving the same $T$-factorization are already equal up to $R$-associates and order. Hence $|\mathsf Z_R(r)| \le |\mathsf Z_T(r)| < \infty$. Together with the descent argument of the first paragraph, this proves (2).
	\end{proof}
	
	For a coefficient subring that is not a field, the following local case gives a different criterion.
	
	\begin{cor}\label{cor:DplusM}
		Let $T=K+M$ be a local integral domain with coefficient field $K$ and nonzero maximal ideal $M$. Let $D$ be a subring of $K$ that is not a field, and set $R=D+M$. Then $R$ is a U-FFD if and only if $D$ is a U-FFD.
	\end{cor}
	
	\begin{proof}
		Choose a nonzero nonunit $d \in D$. Every nonzero $m \in M$ has the decomposition $m = d(m/d)$ into two nonunits of $R$, so $M$ contains no atoms of $R$. Every element $a+m \in R$ with $a\ne0$ can be written as $a(1+m/a)$, and $1+m/a \in R^\times$. Hence every atom of $R$ is associate to an atom of $D$. Two nonzero elements of $D$ are associate in $R$ exactly when they are associate in $D$, as follows by taking residues. Taking residues therefore identifies the factorizations of $a\in D^*$ in $D$ with those in $R$. Elements of $M \setminus \{0\}$ are not atomic, since residues of products of atoms are nonzero. Thus, $R$ is a U-FFD if and only if $D$ is a U-FFD.
	\end{proof}
	
	We proceed to discuss some related examples.
	
	\begin{example} \label{ex:discrete-finite}
		Let $k\subseteq K$ be a field extension. The group $K^\times/k^\times$ is finite when $K$ is finite or $K = k$. If $k$ is infinite and $K \ne k$ then the quotient group $K^\times/k^\times$ is infinite: for $a \in K \setminus k$, the elements $a+c$ with $c \in k$ represent distinct cosets. Now set $T  :=  K \ldb x \rdb = K+M$, where $M$ is the maximal ideal $xK\ldb x \rdb$ of $T$ and then take $R := k+M$. Observe that $M$ contains irreducible elements (those associates of $x$ in $T$). Now we consider the following cases.
		\smallskip
		
		\textsc{Case 1:} Assume that $|K| = q^m$ and $|k| = q$, where $q$ is a prime power and $m \in \nn$. Then $K \cong \mathbb{F}_{q^m}$ and $k \cong \mathbb{F}_q$. Since $T = K\ldb x \rdb$ is a UFD, it is also a U-FFD. Moreover,
		\[
			K^\times/k^\times \cong \mathbb{F}_{q^m}^\times/\mathbb{F}_q^\times
		\]
		is a finite group of order $(q^m-1)/(q-1)$. Therefore, Theorem~\ref{thm:UFF-DplusM} ensures that $R = k + xK\ldb x \rdb$ is a U-FFD.
		\smallskip
		
		\textsc{Case 2:} $k := \mathbb{Q}$ and $K := \rr$. In this case, the quotient group $\rr^\times/\qq^\times$ has infinite cardinality, and therefore we confirm what we argued in Example~\ref{ex:uff_localization_fails}: $R = \qq + x \rr\ldb x\rdb$ is not a U-FFD even though $T$ is a U-FFD. \hfill $\blacksquare$
	\end{example}
	
		We conclude this section with an example of a non-discrete valuation, illustrating the case where the maximal ideal $M$ contains no irreducibles. Recall that the Hahn valuation on the Hahn field $k((t^G))$ sends each nonzero series $f = \sum_{\gamma \in G} a_\gamma t^\gamma$ to $\mathsf{v}(f) := \min \supp(f)$; its valuation ring consists of the series of nonnegative valuation, and its maximal ideal consists of those of positive valuation.
	
	\begin{example}\label{ex:nondiscrete}
		Let $k$ be any field and consider the Hahn valuation domain $T$ with coefficients in $k$
		\[
			T := \ k((t^{\mathbb{Q}}))^{\ge 0} \quad = \ \bigg\{\,\sum_{\gamma\in \mathbb{Q}_{\ge 0}} a_\gamma t^\gamma \ :\ 
			\text{well-ordered support},\ a_\gamma\in k\,\bigg\},
		\]
		with maximal ideal $M := \{ f\in T : f=0 \text{ or } \min\supp f>0 \}$ and coefficient (residue) field identified with $K=k$. After setting $T := K+M$, observe that $M$ contains no irreducible elements (as we are in the rank-$1$ non-discrete case), and $T$ is antimatter. Hence $T$ is vacuously a U-FFD. For any subfield $k_0 \subseteq K$, set $R = k_0+M$. Then $R$ has the U-FF property as well by Theorem~\ref{thm:UFF-DplusM}(2); the equality $R=T$ holds exactly when $k_0 = K$. This gives a large class of integral domains where the U-FF property holds automatically on both~$T$ and~$R$. \hfill $\blacksquare$
	\end{example}

\bigskip
\section{Monoid Algebras}
\label{sec:ascent to polynomial rings}

The primary purpose of this section is to study the ascent of the U-FF property to monoid algebras. The most relevant ascent problem in this direction is the special case of polynomial extensions, which we tackle first.

\medskip
\subsection{Ascent to Polynomial Extensions}

The ascent of the IDF property to polynomial extensions was posed as an open question by Anderson, Anderson, and Zafrullah~\cite{AAZ90} back in 1990, and it was proved to be false over the general class of integral domains by Malcolmson and Okoh~\cite{MO09} in 2009. As with the IDF property, it turns out that the U-FF property does not ascend to polynomial extensions. We proceed to discuss an example illustrating this observation.

\begin{example} \label{ex:U-FF does not ascend to polynomial extensions}
	Let $x, y_1, y_2, y_3$ be pairwise distinct indeterminates, and let $(z_n)_{n \ge 0}$ be an infinite sequence of pairwise distinct indeterminates whose underlying set is disjoint from $\{x, y_1, y_2, y_3\}$. Now fix a prime $p \in \pp$ and consider the monoid $M := \nn_0\big[1/p\big]$. The primary ring in this example is
    \[
        R := R_1 + R_2\big[x^q : q \in M_{>0} \big],
    \]
    where
    \[
        R_1 := \mathbb F_p\big[z_n : n \in \nn_0 \big] \quad \text{and} \quad R_2 := \ff_p\Big[ x^q, z_i^q,  \Big( \frac{y_1 x}{z_i}\Big)^q, \Big(\frac{y_2 x}{z_i}\Big)^q, \Big(\frac{y_3 x}{z_i}\Big)^q : (i,q) \in \mathbb N_0 \times M \Big].
    \]
    Since $x^q\in R_2$ for every $q\in M$ and $R_1\subseteq R_2$, the displayed definition gives $R=R_2$. Every element of $R_2$ has a $p$-th root in $R_2$, so $R$ is antimatter and therefore a U-FFD. After reindexing the $z_i$, it is the same ring as in Example~\ref{ex:a IDF that is neither BF nor MCD-finite}. In particular, $R^\times=\ff_p^\times$. Now let $w$ be an indeterminate over $R$, and let us argue that $R[w]$ does not satisfy the U-FF property. For this, it suffices to consider the linear polynomial $H := A+Bw$, where
    \[
        A := x(y_1+y_3) \qquad \text{and} \qquad B := x(y_2+y_3),
    \]
	    and prove that $H^2/(z_1 z_2)$ is a polynomial expression in $R[w]$ having infinitely many length-$2$ factorizations. For each $t \in \zz[1/p] \cap [0,1]$, put
    \[
        G_t := \frac{H}{z_1^t z_2^{1-t}}
    \]
    and observe that $G_t$ belongs to $R[w]$ because
    \[
        \frac{xy_j}{z_1^t z_2^{1-t}}
        = \left(\frac{xy_j}{z_1}\right)^t
          \left(\frac{xy_j}{z_2}\right)^{1-t}
    \]
	for every $j \in \ldb 1,3 \rdb$. In particular, $G_0, G_1 \in R[w]$, and so
	    \[
	        \frac{H^2}{z_1z_2}
	        = \frac{H}{z_2}\frac{H}{z_1}
	        = G_0G_1 \in R[w].
	    \]

    We claim that $G_t$ is irreducible. Let $d\in R$ be a nonzero common divisor of its two coefficients. The first coefficient contains no $y_2$, the second contains no $y_1$, and their difference contains no $y_3$. As degrees in each $y_j$ are nonnegative and additive under multiplication, $d$ involves none of $y_1,y_2,y_3$. Homogeneity of the first coefficient in $x$ and in each $z_i$, together with additivity of the maximum and minimum exponents in a product, then gives
    \[
        d = \alpha x^r\prod_{i\in\nn_0}z_i^{s_i}
    \]
    for some $\alpha\in\ff_p^\times$ and $r,s_i\in M$, with only finitely many $s_i$ nonzero. As in Example~\ref{ex:a IDF that is neither BF nor MCD-finite}, every monomial of $R$ has $x$-exponent at least its total $y$-degree, and the sum of the absolute values of its negative $z_i$-exponents is at most its total $y$-degree. Applying these inequalities to a monomial of either coefficient divided by $d$ gives
    \[
        1-r\ge 1 \qquad\text{and}\qquad 1+\sum_{i\in\nn_0}s_i\le 1.
    \]
    Therefore $r=s_i=0$ for every $i\in\nn_0$, so $d$ is a unit. Since $G_t$ is linear in $w$, this proves that $G_t$ is irreducible in $R[w]$, as desired.

	As a consequence of the established claim, for each $t\in\zz[1/p]\cap[0,1]$, we obtain the following factorization of $H^2/(z_1 z_2)$ in $R[w]$:
    \[
        \frac{H^2}{z_1z_2} = G_tG_{1-t}.
    \]
    Since $R[w]^\times = \ff_p^\times$, the elements $G_t$ and $G_u$ are associates only when $t=u$. Hence two of these factorizations coincide up to order and associates only when their parameters are equal or sum to $1$. Thus, $H^2/(z_1z_2)$ has infinitely many length-$2$ factorizations, and we can conclude that $R[w]$ is not a U-FFD. \hfill $\blacksquare$
\end{example}
	
The integral domain $R$ in the previous example is antimatter and hence satisfies the U-UF property. Its polynomial ring is not a U-FFD and therefore is not a U-UFD. We conclude this section with the following remark.

\begin{rem}
    The U-UF property does not ascend to polynomial extensions.
\end{rem}

Additional conditions on $R$ so that the IDF property ascends from $R$ to its polynomial extensions have been established in \cite{MO09,mZ17,GZ23}. The most general condition was proposed by Eftekhari and Khorsandi in~\cite{EK18}: they proved that if $R$ is an integral domain with the MCD-finite condition then the IDF property does ascend to polynomial extensions. The same holds for the U-FF property, as we proceed to prove.

\begin{theorem}\label{thm:U-FFD-ascends-to-poly}
	Let $R$ be an MCD-finite domain. If $R$ is a U-FFD then $R[x]$ is also a U-FFD.
\end{theorem}

\begin{proof}
	Let $K$ be the field of fractions of $R$. Fix a nonzero $f\in R[x]$ that is atomic in $R[x]$. We first show that $f$ has only finitely many non-associate irreducible divisors of positive degree.
	
	Since $K[x]$ is a UFD, it has only finitely many divisor classes of $f$ up to $K[x]$-associates. If $f$ had infinitely many non-associate irreducible divisors of positive degree in $R[x]$, we could choose infinitely many of them, say $(g_n)_{n \ge 1}$, in a single $K[x]$-associate class. Thus,
	\[
		\frac{g_n}{\operatorname{lc}(g_n)} = \frac{g_1}{\operatorname{lc}(g_1)}
	\]
	for every $n \in \nn$. Now set $h := \operatorname{lc}(f)g_1$ and, after fixing $n \in \nn$, set
	\[
		d_n := \frac{\operatorname{lc}(f)}{\operatorname{lc}(g_n)}\operatorname{lc}(g_1) \in R.
	\]
	Then $h = d_n  g_n$ and, because $g_n$ is irreducible of positive degree, its coefficients have no common nonunit divisor, which implies that $d_n$ is an MCD of the nonzero coefficients of $h$. If $d_n$ and $d_m$ were associates in~$R$ then $g_n$ and $g_m$ would also be associates in $R[x]$. We therefore obtain infinitely many non-associate MCDs of a fixed finite set, contrary to the hypothesis.
	
	Choose representatives $h_1, \dots, h_t$ for the positive-degree irreducible divisor classes of $f$. Every factorization of $f$ contains at most $\deg f$ factors of positive degree, so only finitely many multisets of such factors can occur. Fix one such multiset and let $h$ be the product of its chosen representatives. In every factorization with this multiset, after absorbing the units, the remaining constant atoms have product $a = f/h \in R$. The constant atoms of $R[x]$ are exactly the atoms of $R$, and $a$ has only finitely many factorizations in~$R$ because $R$ has the U-FF property. Hence each multiset gives only finitely many factorizations of $f$. As there are finitely many multisets, $f$ has only finitely many factorizations, proving that $R[x]$ is a U-FFD.
\end{proof}
	
\medskip
\subsection{Ascent to Monoid Algebras over Fields}
	
As it the case of the IDF property, the U-FF property does not ascend in general from a cancellative torsion-free monoid $M$ to its monoid algebras over fields. In the following example we construct a U-FFM whose monoid algebra over $\qq$ is not a U-FFD.
	
\begin{example}\label{ex:rank-four-UFF-does-not-ascend}
	First, we produce a reduced torsion-free antimatter (and, therefore, trivially U-FF) monoid $M$ of rank $4$. Then we argue that the monoid algebra of $M$ over $\qq$ does not satisfy the U-FF property. Set $D := \zz[1/2]$, and for each $r\in D$, write $r^- := \max\{0,-r\}$. Consider the additive monoid
	\[
		M := \big\{(a,b,c,d)\in D^4:
		a,b\ge 0 \text{ and } c^-+d^-\le a+b\big\}.
	\]
	The inequality $(r+s)^-\le r^-+s^-$ ensures that $M$ is closed under addition. Also, if both $m$ and $-m$ belong to $M$ then all four coordinates of $m$ are zero, whence $M$ is reduced. As a submonoid of $D^4$, it is cancellative and torsion-free. Moreover, $M$ contains the four standard basis vectors, so its rank is~$4$. In addition, observe that, for each nonzero $m \in M$, the element $m/2$ is a nonzero element of $M$ and so $M$ is antimatter. In particular, $M$ satisfies the U-FF property.
	\smallskip
	
	Let us now consider the monoid algebra $R :=\qq[M]$ of $M$ over $\qq$. We can identify $R$ with a subring of the group algebra $\qq[D^4]$ as follows:
	\[
		R = \bigoplus_{(a,b,c,d) \in M}
		\qq \, x^ay^bz^cw^d,
	\]
	where $x,y,z,w$ are independent formal variables. Since $M$ is reduced and torsion-free, $R^\times = \qq^\times$. For each $t \in D \cap [0,1]$, set
	\[
		f_t := \frac{x+y}{z^tw^{1-t}}.
	\]
	The exponent vectors of the two monomials of $f_t$ are $(1, 0, -t, t-1)$ and $(0,1,-t,t-1)$, both of which belong to $M$. Thus, $f_t \in R$. We proceed to establish the following claim.
	\smallskip
	
	\noindent \textsc{Claim.} The element $f_t$ is irreducible in $R$ for every $t \in D \cap [0,1]$.
	\smallskip
	
	\noindent \textsc{Proof of Claim.} Fix $t \in D \cap [0,1]$, and suppose that $f_t = gh$ for some $g,h \in R$. Since all exponents involved are dyadic rationals, this factorization takes place in the Laurent polynomial ring
	\[
		L_N	:=	\qq	\big[x^{\pm1/2^N},y^{\pm1/2^N}, z^{\pm1/2^N},w^{\pm1/2^N}\big]
	\]
	for some sufficiently large $N\in\nn_0$. Writing $X=x^{1/2^N}$ and $Y=y^{1/2^N}$, we see that $f_t$ is associate in $L_N$ to $X^{2^N}+Y^{2^N}$. This polynomial is irreducible in $\qq[X,Y]$: it is the homogenization of the irreducible cyclotomic polynomial
	\[
		\Phi_{2^{N+1}}(T)=T^{2^N}+1.
	\]
	It remains irreducible after adjoining the other variables and inverting the variables, as it is not divisible by any variable. Consequently, one of $g$ and $h$ is a unit of $L_N$, and hence a Laurent monomial.
	
	It remains to show that every monomial divisor of $f_t$ in $R$ is a unit of $R$. Suppose that
	\[
		u = \lambda x^ay^bz^cw^d \mid_R f_t,
	\]
	for some  $\lambda \in \qq^\times$ and $(a,b,c,d) \in M$. Since $f_t/u \in R$, the exponent vectors $(1-a, -b, -t-c, t-1-d) \quad \text{and} \quad (-a,1-b,-t-c,t-1-d)$ both belong to $M$. The nonnegativity of their first two coordinates, together with $a,b\ge 0$, yields $a=b=0$. The fact that $(0,0,c,d)\in M$ then implies $c,d\ge 0$. Applying the defining inequality of $M$ to either exponent vector of $f_t/u$, we obtain $(t+c) + (1-t+d) \le 1$. Thus, $c=d=0$, and $u=\lambda\in R^\times$. Thus, the claim is established.
	\smallskip
	
	In order to show that $\qq[M]$ does not satisfy the U-FF property it suffices to exhibit a polynomial expression $F$ with infinitely many factorizations, and our candidate is
	\[
		F := \frac{(x+y)^2}{zw} = \bigg(\frac{x+y}{w} \bigg) \bigg( \frac{x+y}{z}\bigg) = f_0 f_1 \in R.
	\]
	To argue that $F$ has infinitely many factorizations, first observe that
	\[
		f_t f_{1-t} =  \bigg(\frac{x+y}{z^tw^{1-t}} \bigg) \bigg( \frac{x+y}{z^{1-t}w^t}\bigg) = \frac{(x+y)^2}{zw} = F.
	\]
	Since $R^\times = \qq^\times$, the elements $f_s$ and $f_t$ are non-associate in $R$ for any distinct $s,t \in D \cap [0,1]$. After writing $[f_t]$ for the associate class of $f_t$, we can see that the formal products $[f_t][f_{1-t}]$ with $t\in D\cap[0,1/2]$ are pairwise distinct elements of $\mathsf Z_R(F)$. Hence~$F$ has infinitely many factorizations in~$R$. Hence $\qq[M]$ is not a U-FFD, even though $M$ is an antimatter U-FFM.
	\hfill $\blacksquare$
\end{example}

	\bigskip
	\section{A Characterization of the Finite Factorization Property}
	\label{sec:a characterization of FFDs via near atomicity}
	
	It is well known that an integral domain has the FF property if and only if it is an atomic IDF domain. It was recently proved by the authors and Hong~\cite{DGH26} that we can replace the condition of being atomic by that of being nearly atomic and still obtain an equivalence. As the following proposition shows, one can go a step further: replacing the IDF property with the U-FF property in~\cite[Theorem~4.2]{DGH26}, one obtains another equivalent condition for the FF property.
	
	\begin{prop} \label{prop:a characterization of FFDs via near atomicity}
		For a cancellative and commutative monoid $M$, the following conditions are equivalent.
		\begin{enumerate}
			\item[(a)] $M$ is an FFM.
			\smallskip
			
			\item[(b)] $M$ is a nearly atomic U-FFM.
		\end{enumerate}
	\end{prop}
	
	\begin{proof}
		Given that a monoid is atomic (resp., nearly atomic, a U-FF monoid, an FFM) if and only if its reduced monoid $M/\uu(M)$ is atomic (resp., nearly atomic, a U-FF monoid, an FFM), we can assume without loss of generality that the monoid $M$ is reduced.
		\smallskip
		
		(a) $\Rightarrow$ (b): Assume that $M$ is an FFM. By definition, $M$ is atomic and, therefore, $M$ is nearly atomic. In addition, it follows from the FFM characterization recalled in Section~\ref{sec:background} that $M$ is an IDF monoid, and hence satisfies the U-FF property.
		\smallskip
		
		(b) $\Rightarrow$ (a): Assume now that $M$ is a nearly atomic U-FF monoid. Because $M$ is nearly atomic, one can take a shift $s \in M$ such that every element of the principal ideal $sM$ is atomic. We now prove the following claim.
		\smallskip
		
		\noindent \textsc{Claim.} The monoid $M$ satisfies the ACCP.
		\smallskip
		
		\noindent \textsc{Proof of Claim.} Suppose by way of contradiction that $M$ violates the ACCP. Thus, there must exist a sequence $(a_n)_{n \ge 1}$ of pairwise nonassociate elements of $M$ such that $a_{n+1} \mid_M a_n$ for every $n \in \mathbb N$. 
		
		However, for any $n\in \mathbb N$, note that $$s^2a_1 = (s a_n)\left(\frac{sa_1}{a_n}\right)$$ splits the atomic element $s^2a_1$ into the product of two atomic cofactors: $sa_n$ and $ \frac{sa_1}{a_n}$. As $s^2a_1$ is an atomic element of the U-FFM $M$, it has only finitely many atomic factorizations. Furthermore, each factorization can only be split into two cofactors in finitely many ways. But since all $sa_n$ are nonassociate, we identify infinitely many distinct ways to split cofactors, yielding contradiction. Thus, the claim is established.
		\smallskip
		
		As a result, it follows that $M$ is an atomic monoid. Now the fact that $M$ is an atomic U-FF monoid implies that $M$ is an FFM, which concludes the proof.
	\end{proof}
	
	As an immediate consequence, we obtain the following.
	
	\begin{cor}
		An integral domain is an FFD if and only if it is nearly atomic and has the U-FF property.
	\end{cor}
	
	The most natural weaker notion of near atomicity is that of almost atomicity introduced by Boynton and Coykendall in~\cite{BC15}. The following example of an almost atomic IDF (and so U-FF) domain that is not an FFD was constructed in~\cite{DGH26}.
	
	\begin{example}
		Let $R = \mathbb R[x,y; M]$ be the monoid algebra of the submonoid $M = (\mathbb N_0 \times \mathbb N_0) \cup (\mathbb Z \times \mathbb N_{\ge 2})$ of the free abelian group $\zz^2$ with real coefficients. It turns out that~$R$ is an almost atomic IDF domain that is not atomic. Thus, $R$ is not an FFD.
		\hfill $\blacksquare$
	\end{example}

	\bigskip
	\section*{Acknowledgments}
	
	During the period of this collaboration, the first author was part of PRIMES-USA, a year-long math research program hosted by the MIT Mathematics department. The authors thank the PRIMES program for this rewarding research opportunity. During the collaboration leading to this paper, the second author was partially supported by the NSF under the award DMS-2213323.

\end{document}